\documentclass[reqno]{amsart}

\usepackage[all]{xy}
\usepackage{graphicx}

\usepackage{amssymb}
\usepackage{amsmath}
\usepackage{mathrsfs}
\usepackage{epsfig}
\usepackage{amscd}
\usepackage{graphicx, color}

\def\E{\ifmmode{\mathbb E}\else{$\mathbb E$}\fi} 
\def\N{\ifmmode{\mathbb N}\else{$\mathbb N$}\fi} 
\def\R{\ifmmode{\mathbb R}\else{$\mathbb R$}\fi} 
\def\Q{\ifmmode{\mathbb Q}\else{$\mathbb Q$}\fi} 
\def\C{\ifmmode{\mathbb C}\else{$\mathbb C$}\fi} 
\def\H{\ifmmode{\mathbb H}\else{$\mathbb H$}\fi} 
\def\Z{\ifmmode{\mathbb Z}\else{$\mathbb Z$}\fi} 
\def\P{\ifmmode{\mathbb P}\else{$\mathbb P$}\fi} 
\def\T{\ifmmode{\mathbb T}\else{$\mathbb T$}\fi} 
\def\SS{\ifmmode{\mathbb S}\else{$\mathbb S$}\fi} 
\def\DD{\ifmmode{\mathbb D}\else{$\mathbb D$}\fi} 

\newcommand{\del}{\partial}
\newcommand{\Cont}{{\operatorname{Cont}}}

\newcommand{\ben}{\begin{enumerate}}
\newcommand{\een}{\end{enumerate}}
\newcommand{\be}{\begin{equation}}
\newcommand{\ee}{\end{equation}}
\newcommand{\bea}{\begin{eqnarray}}
\newcommand{\eea}{\end{eqnarray}}
\newcommand{\beastar}{\begin{eqnarray*}}
\newcommand{\eeastar}{\end{eqnarray*}}
\newcommand{\bc}{\begin{center}}
\newcommand{\ec}{\end{center}}

\newcommand{\norm}[2]{{ \ensuremath{\|} #1 \ensuremath{\|}}_{#2}}
\theoremstyle{theorem}
\newtheorem{thm}{Theorem}[section]
\newtheorem{cor}[thm]{Corollary}
\newtheorem{lem}[thm]{Lemma}
\newtheorem{prop}[thm]{Proposition}

\theoremstyle{definition}
\newtheorem{defn}[thm]{Definition}
\newtheorem{rem}[thm]{Remark}
\newtheorem{ques}[thm]{Question}

\newtheorem{hypo}[thm]{Hypothesis}

\newtheorem*{thm*}{Theorem}

\numberwithin{equation}{section}

\def\R{{\mathbb R}}

\def\E{{\mathbb E}}
\def\Z{{\mathbb Z}}
\def\C{{\mathbb C}}
\def\R{{\mathbb R}}
\def\P{{\mathbb P}}

\def\N{{\mathbb N}}

\def\11{{\mathbb I}}

\def\delbar{{\bar \partial}}

\def\C{\mathbb{C}}
\def\Z{\mathbb{Z}}

\def\T{\mathbb{T}}

\def\Q{\mathbb{Q}}

\def\E{\ifmmode{\mathbb E}\else{$\mathbb E$}\fi} 
\def\N{\ifmmode{\mathbb N}\else{$\mathbb N$}\fi} 
\def\R{\ifmmode{\mathbb R}\else{$\mathbb R$}\fi} 
\def\Q{\ifmmode{\mathbb Q}\else{$\mathbb Q$}\fi} 
\def\C{\ifmmode{\mathbb C}\else{$\mathbb C$}\fi} 
\def\H{\ifmmode{\mathbb H}\else{$\mathbb H$}\fi} 
\def\Z{\ifmmode{\mathbb Z}\else{$\mathbb Z$}\fi} 
\def\P{\ifmmode{\mathbb P}\else{$\mathbb P$}\fi} 
\def\SS{\ifmmode{\mathbb S}\else{$\mathbb S$}\fi} 
\def\DD{\ifmmode{\mathbb D}\else{$\mathbb D$}\fi} 

\def\R{{\mathbb R}}

\def\E{{\mathbb E}}
\def\Z{{\mathbb Z}}
\def\C{{\mathbb C}}
\def\R{{\mathbb R}}

\def\N{{\mathbb N}}
\def\MM{{\mathcal M}}

\def\JJ{{\mathcal J}}

\def\delbar{{\overline \partial}}

\def\CC{{\mathcal C}}
\def\CD{{\mathcal D}}

\def\CF{{\mathcal F}}

\def\CH{{\mathcal H}}

\def\CJ{{\mathcal J}}

\def\CL{{\mathcal L}}
\def\CM{{\mathcal M}}

\def\CS{{\mathcal S}}

\def\CW{{\mathcal W}}

\def\darr#1{\raise1.5ex\hbox{$\leftrightarrow$}
\mkern-16.5mu #1}

\def\roughly#1{\raise.3ex\hbox{$#1$\kern-.75em
\lower1ex\hbox{$\sim$}}}

\def\opname#1{\mathop{\kern0pt{\rm #1}}\nolimits}
\def\Re{\opname{Re}}
\def\Im{\opname{Im}}
\def\End{\opname{End}}
\def\dim{\opname{dim}}

\def\Area{\opname{Area}}
\def\genus{\opname{genus}}

\def\supp{\operatorname{supp}}

\def\End{\operatorname{End}}

\def\span{\operatorname{span}}

\begin{document}
\quad \vskip1.375truein

\def\mq{\mathfrak{q}}
\def\mp{\mathfrak{p}}
\def\mH{\mathfrak{H}}
\def\mh{\mathfrak{h}}
\def\ma{\mathfrak{a}}
\def\ms{\mathfrak{s}}
\def\mm{\mathfrak{m}}
\def\mn{\mathfrak{n}}
\def\mz{\mathfrak{z}}
\def\mw{\mathfrak{w}}
\def\Hoch{{\tt Hoch}}
\def\mt{\mathfrak{t}}
\def\ml{\mathfrak{l}}
\def\mT{\mathfrak{T}}
\def\mL{\mathfrak{L}}
\def\mg{\mathfrak{g}}
\def\md{\mathfrak{d}}
\def\mr{\mathfrak{r}}

\title[Contact Cauchy-Riemann maps III]{Analysis of contact Cauchy-Riemann maps III:
energy, bubbling and Fredholm theory}

\author{Yong-Geun Oh}
\address{Center for Geometry and Physics, Institute for Basic Science (IBS),
77 Cheongam-ro, Nam-gu, Pohang-si, Gyeongsangbuk-do, Korea 790-784
\& POSTECH, Gyeongsangbuk-do, Korea}
\email{yongoh1@postech.ac.kr}
\thanks{This work is supported by the IBS project \# IBS-R003-D1}

\date{}

\begin{abstract}  In \cite{oh-wang2,oh-wang3}, the authors studied the nonlinear elliptic system
$$
\delbar^\pi w = 0, \, d(w^*\lambda \circ j) = 0
$$
\emph{without involving symplectization}
for each given contact triad $(Q,\lambda, J)$, and established the a priori
 $W^{k,2}$ elliptic estimates and proved
the asymptotic (subsequence) convergence of the map $w: \dot \Sigma \to Q$
for any solution, called a contact instanton, on $\dot \Sigma$ under the hypothesis
$\|w^*\lambda\|_{C^0} < \infty$ and $d^\pi w \in L^2 \cap L^4$. The asymptotic limit of
a contact instanton is a `spiraling' instanton along a `rotating' Reeb orbit
near each puncture on a punctured Riemann surface $\dot \Sigma$. Each limiting Reeb orbit
carries a `charge' arising from the integral of $w^*\lambda \circ j$.

In this article, we further develop analysis of contact instantons, especially
the $W^{1,p}$ estimate for $p > 2$ (or the $C^1$-estimate), which is essential
for the study of compactfication of the moduli space
and the relevant Fredholm theory for contact instantons.
In particular, we define a Hofer-type off-shell energy $E^\lambda(j,w)$ for any pair $(j,w)$
with a smooth map $w$ satisfying  $d(w^*\lambda \circ j) = 0$, and develop the bubbling-off analysis
and prove an $\epsilon$-regularity result. We also develop the relevant Fredholm theory and
carry out index calculations (for the case of vanishing charge).
\end{abstract}



\maketitle

\tableofcontents

\section{Introduction and statements of main results}
\label{sec:intro}

A contact manifold $(Q,\xi)$ is a $2n+1$ dimensional manifold
equipped with a completely non-integrable distribution of rank $2n$,
called a contact structure. Complete non-integrability of $\xi$
can be expressed by the non-vanishing property
$$
\lambda \wedge (d\lambda)^n \neq 0
$$
for a one-form $\lambda$ which
defines the distribution, i.e., $\ker \lambda = \xi$. Such a
one-form $\lambda$ is called a contact form associated to $\xi$.
Each contact form $\lambda$ of $\xi$ canonically induces a splitting
$$
TQ = \R\{X_\lambda\} \oplus \xi.
$$
Here $X_\lambda$ is the Reeb vector field of $\lambda$,
which is uniquely determined by the equations
$$
X_\lambda \rfloor \lambda \equiv 1, \quad X_\lambda \rfloor d\lambda \equiv 0.
$$
We denote by $\Pi=\Pi_\lambda: TQ \to TQ$ the idempotent, i.e., an endomorphism satisfying
$\Pi^2 = \Pi$ such that $\ker \Pi = \R\{X_\lambda\}$ and $\operatorname{Im} \Pi = \xi$.
Denote by $\pi=\pi_\lambda: TQ \to \xi$ the associated projection.

In the presence of the contact form $\lambda$,
one usually consider the set of $J$ that is compatible to $d\lambda$ in
the sense that the bilinear form $g_\xi = d\lambda(\cdot, J\cdot)$ defines
a Hermitian vector bundle $(\xi, d\lambda|_\xi,J|_\xi)$ on $Q$. We call such $J$ a
$CR$-almost complex structure.
As long as no confusion arises, we do not distinguish $J$ and its restriction $J|_\xi$.
We introduce the projection $\pi: TQ \to \xi$ with respect to the
splitting $TQ = \R\{X_\lambda\} \oplus \xi$.

\begin{defn} Let $J \in \End(TQ)$ be an endomorphism satisfying $J^2 = - \Pi$ such that
$d\lambda(\cdot,J\cdot)$ is nondegenerate on $\xi$. We say that such $J$ is \emph{compatible to} $\lambda$.
 We define the set
\be\label{eq:JJQlambda}
\CJ(Q,\lambda) = \{J :\xi \to \xi \mid J^2 = - \Pi, \, J \, \text{ compatible to } \, \lambda\}
\ee
\end{defn}
Following \cite{oh-wang1}, we call any such triple $(Q,\lambda,J)$ a contact triad of $(Q,\xi)$.
For each given contact triad, we equip $Q$ with the triad metric
$$
g = d\lambda(\cdot, J \cdot) + \lambda \otimes \lambda.
$$

Let $(\Sigma, j)$ be a Riemann surface with a finite number of marked points and
let $\dot \Sigma$ be the associated punctured Riemann surface with
a finite number of punctures. We call a map $w:\dot \Sigma \to Q$ a
\emph{contact Cauchy-Riemann map} if $\delbar^\pi w = 0$.
Then we have the decomposition
$$
dw = d^\pi w + w^*\lambda\, X_\lambda, \quad  d^\pi w : = \delbar^\pi w + \del^\pi w
$$
as a one-form on $\Sigma$ with values in $TQ$. We also regard $d^\pi w$ as
a $\xi$-valued one-form on $\Sigma$.

We introduce a nonlinear first-order differential operator
\be\label{eq:dw-decompose}
\delbar^\pi w = \frac{1}{2} (\pi dw + J \cdot\pi dw \cdot j), \,
\del^\pi w= \frac{1}{2} (\pi dw - J \cdot\pi dw \cdot j)
\ee
and consider the following variation of Cauchy-Riemann equation
\be\label{eq:CRcontact}
\delbar^\pi w = 0.
\ee
\begin{defn}\label{defn:CRcontact}
We say a map $w: \Sigma \to Q$ is a contact Cauchy-Riemann map
(with respect to $J$) if it satisfies \eqref{eq:CRcontact}.
\end{defn}

In \cite{oh-wang2}, Wang and the present author established the a priori
$W^{k,2}$ coercive estimates for the contact Cauchy-Riemann maps by
augmenting the equation $\delbar^\pi w = 0$ by the closedness condition of
\be\label{eq:closed}
d(w^*\lambda \circ j) = 0.
\ee
The standard pseudoholomorphic curve equation
on the symplectization $Q \times \R$ equipped with the cylindrical
almost complex structure $J_0 \oplus J$ with respect to the splitting
$$
T(Q \times \R) =  \xi \oplus \R\cdot X_\lambda \oplus \R \cdot \frac{\del}{\del r}
$$
is a special case of the `exact' contact instantons where the anti-derivative
equation of $w^*\lambda\circ j$ prescribed by $a = w^*s$ with the
$s$-coordinate of the symplectization $Q \times \R$
for the map $(w,a): \dot \Sigma \to Q \times \R$. (See \cite{hofer} for the
relevant calculations.)

\begin{defn}[Contact instanton] Let $\Sigma$ be as above.
We call a pair of $(j,w)$ of a complex structure on $\Sigma$ and a map $w:\dot \Sigma \to Q$ a
\emph{contact instanton} if it satisfies
\be\label{eq:contacton}
\delbar^\pi w = 0, \, \quad d(w^*\lambda \circ j) = 0.
\ee
We call such $(j,w)$ an \emph{exact contact instanton} if the form $w^*\lambda \circ j$ is exact
on $\dot \Sigma$.
\end{defn}
Such an equation was first introduced by Hofer in \cite{hofer-survey} for the case of
charges vanishing at the punctures in the context of symplectization, which was further
studied in \cite{ACH}, \cite{bergmann} and \cite{abbas}. We will also put this charge vanishing
condition at the punctures for our study of the exponential convergence and of the Fredholm
theory at least in the present paper, without involving the symplectization.

To put the research performed in the present paper in perspective, we recall
the precise statement of the above mentioned a priori $W^{k,2}$ estimates established in
\cite{oh-wang2} on the punctured Riemann surface $\dot \Sigma$ here. Denote
$$
w^*\lambda = a^w_1\, d\tau + a^w_2\, dt.
$$

\begin{thm}[Theorem 1.9 \cite{oh-wang2}]\label{thm:Wk2} Let $(\dot\Sigma,j)$ and $w$
satisfying \eqref{eq:contact-instanton} on $\dot \Sigma$ as above.
If $|d^\pi w| \in L^2 \cap L^4$ and $\|w^*\lambda\|_{C^0}< \infty$ on $\dot \Sigma$, then
$$
\int_{\dot \Sigma} |(\nabla)^{k+1}(dw)|^2 \leq \int_{\dot \Sigma} J_{k}'(d^\pi w, w^*\lambda).
$$
Here $J'_{k+1}$ a polynomial function of the norms of the covariant derivatives
of $d^\pi w, \, w^*\lambda$ up to $0, \, \ldots, k$ with degree at most $2k + 4$
whose coefficients depend on
$$
\|K\|_{C^k}, \|R^\pi\|_{C^k}, \|\CL_{X_\lambda}J\|_{C^k}, \, \|w^*\lambda\|_{C^0}.
$$
\end{thm}

One novel feature of this estimate is its explicit reliance on the $C^0$ bound of $w^*\lambda$
which concerns the $X_\lambda$ component of $dw$. Therefore the remaining task is to complete
the a priori estimates to study compactness properties of the moduli space of contact
instantons is to further analyze how to control the quantities
$$
\|w^*\lambda\|_{C^0}, \quad \|d^\pi w\|_{L^4}.
$$
\subsection{Bubbling-off analysis and $\epsilon$-regularity theorem}
\label{subsec:bubbling}

One of the main purposes of the present article is to establish
the two crucial analytical components in the construction of cementification of
the moduli space of solutions of the contact instantons, one the $\epsilon$-regularity theorem
and the other the bubbling-off analysis.

To state the $\epsilon$-regularity statement relevant to contact instantons, we
recall the following standard quantity in contact geometry

\begin{defn} Let $\lambda$ be a contact form of contact manifold $(Q,\xi)$.
Denote by $\frak R eeb(Q,\lambda)$ the set of closed Reeb orbits.
We define $\operatorname{Spec}(Q,\lambda)$ to be the set
$$
\operatorname{Spec}(Q,\lambda) = \left\{\int_\gamma \lambda \mid \lambda \in \frak Reeb(Q,\lambda)\right\}
$$
and call the \emph{action spectrum} of $(Q,\lambda)$. We denote
$$
T_\lambda: = \inf\left\{\int_\gamma \lambda \mid \lambda \in \frak Reeb(Q,\lambda)\right\}.
$$
\end{defn}

We set $T_\lambda = \infty$ if there is no closed Reeb orbit. This set a priori could be empty.
The Weinstein conjecture is equivalent to the statement that this set is non-empty on any compact
contact manifold. A standard lemma in
contact geometry says that $T_\lambda > 0$.
This constant $T_\lambda$ enters in a crucial way in the following
$\epsilon$-regularity type statement. In addition, we also need a Hofer-type energy, denoted by
$E^\lambda(w)$ whose precise definition we refer readers to section \ref{sec:offshellenergy}.

\begin{thm}\label{thm:period-gap}
Denote by $D$ the closed disc of positive radius.
Suppose that $w:D\to Q$ satisfies $\delbar^\pi w = 0, \, d(w^*\lambda\circ j) = 0$ with $E^\lambda(w):=K_0 < \infty$.
Then for any $\epsilon > 0$
and another smaller disc $D' \subset \overline D' \subset D$,
there exists some $K_1 = K_1(p,D', \epsilon,K_0) > 0$ such that for any contact instanton
with $E^\pi(w) < T_\lambda- \epsilon$
\be\label{eq:W2p}
\|dw\|_{1,p;D'} \leq K_1
\ee
where $K_1$ depends only on $p$, $\epsilon$, and $D' \subset D$ and $K_0 = E^\lambda(w)$.
\end{thm}

The proof of this theorem follows the scheme of the corresponding result
in the study of pseudoholomorphic curves given by the author in \cite{oh:removal}. This proof uses
the Sacks-Uhlenbeck's bubbling-off argument which essentially uses the a priori
coercive $W^{k,p}$ elliptic estimates and conformal
invariance of harmonic energy. In the current case of contact instanton maps,
the relevant coercive estimate was established in \cite{oh-wang2}. On the other hand
the harmonic energy is quite irrelevant but the $\pi$-harmonic energy $E^\pi(w)$ is.
However the $\pi$-harmonic energy does not have much control of the derivative
$dw$ in the Reeb direction. In the case of symplectization, Hofer \cite{hofer,behwz} introduced
the so called $\lambda$-energy for the map $u = (w,a): \dot \Sigma \to Q \times \R$
for this purpose. His definition of the
latter energy strongly relies on the coordinate function $a = r\circ w$ which
exists only under the assumption the form $w^*\lambda \circ j$ is exact.
For the non-exact case, we have to devise a different way of defining
Hofer-type $\lambda$-energy. For this purpose, we introduce the notion of \emph{contact instanton potential}
whose definition relies on Zwiebach's representation of conformal structure $j$ on the surface $\dot \Sigma$
by the minimal area metrics \cite{zwiebach,wolf-zwieb}. See section \ref{sec:offshellenergy}
for the details. In the end, our definition of Hofer-type energy strongly depends on the complex
structure $j$ and so had better be regarded as a function for the pair $(j,w)$ not just for $w$.

\subsection{Asymptotic behavior of contact instantons}

We also carry out the
asymptotic study of contact instantons near the punctures.
For this study of asymptotic convergence result at the punctures and the relevant index
theory, it turns out to be useful to regard \eqref{eq:contacton} as a version of
gauged sigma model with abelian Hick's field. It is also important to employ the notion of
asymptotic contact instantons at each puncture,
which is a massless instantons on $\R\times S^1$
canonically associated to any finite energy contact instantons.
It also gives rise to an asymptotic Hick's field, which is a holomorphic one-form
that appears as the asymptotic limit of the complex-valued $(1,0)$-form
$$
w^*\lambda \circ j + \sqrt{-1} w^*\lambda.
$$
The following asymptotic invariants seem to be also useful to introduce in relation to the
precise study of asymptotic behavior of contact instantons near punctures.

\begin{defn}[Asymptotic Hick's charge] Let $(\Sigma, j)$ be a closed Riemann surface
and $\dot \Sigma$ its associated punctured Riemann surface with finite energy with bounded
gradient. Let $p$ be a given puncture of $\dot \Sigma$.
We define the \emph{asymptotic Hick's charge} of the instanton
$w: \dot \Sigma \to Q$ to be the complex number
$$
Q(p) + \sqrt{-1} T(p)
$$
defined by
\bea\label{eq:asymp-charge}
Q(p) & = & -\int_{S^1} \Re \chi(0,t)\, dt = - \int_{\del_{\infty;p}\overline \Sigma}w^*\lambda\circ j \\
T(p) & = & \int_{S^1} \Im \chi(0,t)\, dt = \int_{\del_{\infty;p}\overline \Sigma}w^*\lambda
\eea
where $z = e^{-2\pi(\tau + it)}$ is the analytic coordinates of $D_r(p)$ centered at $p$.
We call $Q(p)$ the \emph{contact instanton charge} of $w$ at $p$
hand $T(p)$ the \emph{contact instanton action} of $w$ at $p$.
\end{defn}
We define the asymptotic Hick's field (or charge) of a map $w: \C \to Q$ at infinity
by regarding $\infty$ as a puncture associated $\C \cong \C P^1 \setminus \{\infty\}$.

We next prove the following removable singularity result (see Theorem \ref{thm:c=0}).

\begin{thm} \label{thm:removable}
Suppose $Q(p) = 0 = T(p)$. Then $w$ is smooth across $p$ and so the
puncture $p$ is removable.
\end{thm}
This theorem will be of fundamental importance in that it enables us to
construct a good compactification of the moduli space of \emph{exact contact instantons}
without involving symplectization. This will be dealt in a sequel to this paper.

The theorem also allows us to make the following classification of the punctures.

\begin{defn}[Classification of punctures] Let $\dot \Sigma$ be a puncture Riemann
surface with punctures $\{p_1, \cdots, p_k\}$ and let
$w: \dot \Sigma \to Q$ be a contact instanton map.
\begin{enumerate}
\item We call a puncture $p$ \emph{removable} if $T(p) = Q(p) = 0$, and \emph{non-removable} otherwise.
Among the non-removable punctures $p$, we call it
\emph{non-adiabatic} if $T(p) \neq 0$, \emph{adiabatic} if $T(p) = 0$ but $Q(p) \neq 0$.
\item
We say a non-removable puncture \emph{positive} (resp. \emph{negative}) puncture if the function
$$
\int_{\del D_\delta(p)} w^*\lambda
$$
is increasing (resp. decreasing) as $\delta \to 0$.
\end{enumerate}
\end{defn}

The appearance of adiabatic punctures is a new phenomenon when the form $w^*\lambda \circ j$ is
not exact. In the exact case considered via the case of symplectization picture \cite{hofer,behwz},
the associated puncture with $T(p) = 0$ is removable and can be dropped in this classification by
filling in the puncture.

Unlike the exact case, the puncture cannot be removed in general
for the non-exact case, i.e., that of non-zero charge $Q(p) \neq 0$, even when $T(p) = 0$. Therefore this
new asymptotic behavior has to be included in the study of moduli space of contact instantons.
What happens at such a puncture is that the instanton $w$ spirals around a leaf of Reeb foliation
when the leaf is closed and chases along the leaf when it is not closed.

We would like to point out the similarity between the relationship of
the forms $w^*\lambda \circ j$ and $w^*\lambda$ for the contact instanton $w$
and the relationship between the electricity and magnetism in the electro-magnetic duality,
in that in both cases the first is associated to the closed one-form while the second is not.
The following  highlights the similarity between the two:
\bea\label{eq:correspondence}
\text{electricity} &\longleftrightarrow& \text{contact instanton charge field } w^*\lambda \circ j \nonumber \\
\text{electric potential} &\longleftrightarrow& \text{contact instanton potential }\, f \nonumber\\
\text{magnetism} &\longleftrightarrow& \text{contact instanton action field }\, w^*\lambda \nonumber\\
&{}&
\eea

\subsection{Triad connection, Fredholm theory and index calculations}

Next we establish the Fredholm theory and compute the index of the
linearization map and hence the virtual dimension of the relevant
moduli space of contact instantons. Establishing the Fredholm theory for the linearization map
$D\Upsilon(w)$ is rather non-trivial because the operator has different
orders depending on the direction of contact distribution $\xi$ or on
the Reeb direction $X_\lambda$ and mixes the directions of the two.
See Theorem \ref{thm:linearization-intro} below.
Our Fredholm theory and its index calculations strongly relies on our
precise calculation of the linearization map via the contact triad connection
introduced in \cite{oh-wang1}. We refer to section \ref{sec:linearization-map} for
the details of the computations.

We denote by $\Sigma$ either the closed Riemann surface or the punctured one.
Recalling the decomposition
$$
Y = Y^\pi + \lambda(Y)\, X_\lambda,
$$
we have the decomposition
$$
\Omega^0(w^*TQ) \cong \Omega^0(w^*\xi) \oplus \Omega^0(\Sigma,\R)\cdot X_\lambda.
$$
Here we use the splitting
$$
TQ =  \xi \oplus \span_\R\{X_\lambda\}
$$
where $\span_\R\{X_\lambda\}: = \CL$ is a trivial line bundle and so
$$
\Gamma(w^*\CL) \cong C^\infty(\Sigma).
$$
Define the map $\Upsilon(w) = (\delbar^\pi w, w^*\lambda\, X_\lambda)$.
From the expression of the map $\Upsilon = (\Upsilon_1,\Upsilon_2)$, the map defines a bounded linear map
\be\label{eq:dUpsilon-intro}
D\Upsilon(w): \Omega^0(w^*TQ) \to \Omega^{(0,1)}(w^*\xi) \oplus \Omega^2(\Sigma).
\ee
We choose $k \geq 2, \, p > 2$. We then establish the following formula

\begin{thm}[Theorem \ref{thm:linearization}] \label{thm:linearization-intro}
Decompose $D\Upsilon(w) = D\Upsilon_1(w) \oplus D\Upsilon_2(w)$
according to the codomain of \eqref{eq:dUpsilon-intro}. Then we have
\bea
D\Upsilon_1(w)(Y) & = & \delbar^{\nabla^\pi}Y^\pi + T_{dw}^{\pi,(0,1)}(Y^\pi) + B^{(0,1)}(Y^\pi)\nonumber \\
&{}& \quad + \frac{1}{2}\lambda(Y) (\CL_{X_\lambda}J)J(\del^\pi w)
\label{eq:DUpsilon1}\\
D\Upsilon_2(w)(Y) & = &  - \Delta (\lambda(Y))\, dA + d((Y^\pi \rfloor d\lambda) \circ j)
\label{eq:DUpsilon2}
\eea
where $T_{dw}^{\pi,(0,1)}$ and $B^{(0,1)}$ are the $(0,1)$-components of $T_{dw}^{\pi}$
and $B$ respectively where $B: \Omega^0(w^*TQ) \to \Omega^1(w^*\xi)$, $T_{dw}^\pi$ are
the zero-order differential operators given by
$$
B(Y) = - \frac{1}{2}  w^*\lambda \left((\CL_{X_\lambda}J)J Y\right)
$$
and
$$
T_{dw}^\pi(Y) = \pi T(Y,dw).
$$
\end{thm}

We denote by $\overline \Sigma$ the real blow-up of the punctured Riemann
surface $\dot \Sigma$ associated to the set of positive and negative punctures
$$
\{p_1, \cdots, p_{s^+}\}, \quad \{q_1, \cdots, q_{s^-}\}
$$
and denote by $\del_i^+\overline \Sigma$ and $\del_j^-\overline \Sigma$ the
associated boundary components. We also denote by $\gamma_i^+$ and $\gamma_j^-$
the given asymptotic Reeb orbits at the punctures.

We fix a trivialization
$\Phi: w^*\xi \to \overline \Sigma \times \R^{2n}$ and denote
by $\Psi_i^+$ (resp. $\Psi_j^-$) the induced symplectic paths associated to the trivializations
$\Phi_i^+$ (resp. $\Phi_j^-$) along the Reeb orbits $\gamma^+_i$ (resp. $\gamma^-_j$) at the punctures
$p_i$ (resp. $q_j$) respectively. Then we have the following index formula
for the case of vanishing charge.  We leave more accurate statements and proof to
section \ref{sec:fredholm}, and the case of non-exact contact instantons elsewhere.

\begin{thm}\label{thm:indexformula} Consider the map $\Upsilon$ defined by
$\Upsilon(w) = (\delbar^\pi w, d(w^*\lambda \circ j))$ on a puncture Riemann surface $\dot \Sigma$.
Let $w$ be an exact contact
instanton, i.e. a solution of $\Upsilon(w) = 0$ with $Q(p_i) = 0$ for all punctures $p_i$.
\begin{enumerate}
\item There exists a compact operator
$$
K: \Omega^0_{k,p}(w^*TQ) \to \Omega^{(0,1)}_{k-1,p;\delta}(w^*\xi)
\oplus \Omega^2_{k-2,p;\delta}(\Sigma)
$$
such that
\beastar
\| D\Upsilon(w) Y\|_{k,p;\delta} & \leq & C (\|D\Upsilon_1(w)(Y)\|_{k-1,p;\delta} +\|\pi_1K(Y)\|_{k-1,p;\delta}\\
&{}& \quad + \|D\Upsilon_2(w)(Y)\|_{k-2,p;\delta} + \|\pi_2K(Y)\|_{k-2,p})
\eeastar
and so the completed map
$$
D\Upsilon(w):
\Omega^0_{k,p;\delta}(w^*TQ) \to \Omega^{(0,1)}_{k-1,p;\delta}(w^*\xi)
\oplus \Omega^2_{k-2,p;\delta}(\Sigma)
$$
is a Fredholm operator if $\delta \in \R \setminus
\CD_w$ for some discrete subset $\CD_w$ of $\R$.
\item Furthermore, provided $0< \delta < \delta_0$ for a sufficiently small $\delta_0$ depending only on $w$,
\beastar
\operatorname{Index} D\Upsilon(w) & = & n(2-2g-s^+ - s^-) + 2c_1(w^*\xi) \\
&{}& \quad  + \sum_{i=1}^{s^+} \mu_{CZ}(\Psi^+_i)
- \sum_{j=1}^{s^-} \mu_{CZ}(\Psi^-_j)\\
&{}& \quad +
\sum_{i=1}^{s^+} (m(\gamma^+_i)+1) + \sum_{j=1}^{s^-}( m(\gamma^-_j)+1) -g\\
\eeastar
where $\mu_{CZ}(\Psi)$ is the Conley-Zehnder index of the symplectic path
$\Psi$ associated to the closed Reeb orbit \cite{conley-zehn,robbin-salamon,hofer}.
\end{enumerate}
\end{thm}

We would like to highlight the appearance of the second line that extracts explicit
contribution depending on the multiplicity of the closed Reeb orbits. Such an appearance
in this kind of index formula seems to be new, at least such an explicit dependence on
the multiplicity does not shows up
in the standard index formula in symplectic field theory such as in Proposition 5.3 \cite{bourgeois}
(with $N = 0$.)

The present paper has been circulated in the author's homepage since year 2013. Only after
the appearance of the paper \cite{oh:contacton-Legendrian-bdy}, which deals with the relevant boundary value
problem is studied, and \cite{oh:entanglement1}, which contains nontrivial application to
the contact Hamiltonian dynamics, the proof of Sandon-Shelukhin's conjecture, we have 
made it public because we now have
enough evidence on usefulness of the analytical machinery of contact instantons
developed in \cite{oh-wang1}, \cite{oh-wang2}, \cite{oh-wang3} and the present paper.
Furthermore we have showed that this machinery is also useful to develop the
theory of pseudoholomorphic curves on locally conformal symplectic manifolds.
(See \cite{oh-savelyev}.)

 The geometric analysis of the contact instantons
such as the $C^1$-convergence and the bubbling analysis of finite energy solutions,
and the derivation of the precise formula of the linearized operator, its Fredholm theory and
relevant tensor calculations developed in the present article are the bases of all later articles
 \cite{oh:contacton-Legendrian-bdy}  -- \cite{oh:perturbed-contacton}
 after combined with their adaptation to the boundary value problems of  the equation.

Besides its naturality of the framework and the aforementioned
intrinsic significance, we now would like to provide further motivation to develop
analysis of contact instantons by comparing it with the existing frameworks of
pseudoholomorphic curves on the symplectization or of the Rabinowitz-Floer homology. The full development of
the analysis of contact instantons and its Hamiltonian perturbations
comprise the two series of papers, the first one consisting of
\cite{oh-wang1}-\cite{oh-wang3} including the present article, and
the second one consisting of \cite{oh:contacton-Legendrian-bdy}-\cite{oh:perturbed-contacton}. One should compare these with
Gromov's pseudoholomorphic curves and Floer's Hamiltonian-perturbed
pseudoholomorphic curves in symplectic geometry, and take the whole package
as a whole  similarly as in symplectic case
when one try to apply  the machinery to problems of contact
dynamics and contact topology.

We would like to emphasize
that \emph{with these analytical foundations of (perturbed) contact instantons
in our  disposal}, the remaining study of contact Hamiltonian dynamics
utilizing perturbed contact instatnons  e.g., construction of relevant
contact spectral invariants on contact manifolds is largely geometro-topological
and dynamical. This enables us to carry out such a study \emph{in an optimal way}
because  the perturbed contact instanton equation \eqref{eq:perturbed-contacton}
interacts with \emph{contact Hamiltonian calculus} best in the straightforward
canonical fashion as illustrated by those in \cite{oh:entanglement1}. Such a study
through the analysis of the perturbed pseudoholomorphic curves such as through Rabinowitz-Floer
homology involves extra step of lifting to the symplectization (see \cite{albers-fuchs-merry} for example),
 which we believe would destroy optimality because of
\emph{irreversibility} of the lifting process. See below for more discussion on this point.

\subsection{Comparison with the analysis of pseudoholomorphic curves on symplectization}

One important feature of our analysis of \eqref{eq:contacton} is that
we do not take symplectization of contact triad $(Q,\lambda,J)$ but directly work
on the contact manifold $Q$.  Hence it enables us to construct
compactification of the smooth moduli space of contact instantons
\emph{with prescribed asymptotic conditions} as long as the charge class is
fixed. This is because \emph{the charge automatically vanishes on the bubbles} since
the domains of bubbles are spheres. (See \cite{oh-savelyev} for the details of
this compactification.) This enables us to define a genuinely contact topological invariant
\emph{without taking the symplectization of $Q$}. Indeed
the question if two contact manifolds having symplectomorphic symplectization are
contactomorphic or not was addressed  in the book by Cieliebak and Eliashberg
\cite{ciel-eliash} and S. Courte \cite{courte} constructed two contact manifolds that have
symplectomorphic symplectization which are not contactomorphic. In this regard,
we hope to investigate the following question stated in \cite{courte} in the future.
\begin{ques} Does there exist contact structures $\xi$ and $\xi'$ on a closed
manifold $M$ that have the same classical invariants and are not contactomorphic,
but whose symplectizations are
(exact) symplectomorphic?
\end{ques}

\subsubsection{Pseudoholomorphic curves in locally conformal symplectic (lcs) manifolds}

As pointed out in \cite{oh-wang1}, the phenomenon of `\emph{appearance of spiraling
contact instantons along the Reeb core}' obstructs the construction of compactified moduli
spaces and its Fredholm theory in general because of
possibility of \emph{nonvanishing of asymptotic charge}. In \cite{oh:contacton-Legendrian-bdy},
this obstacle is automatically removed for the contact instantons with
Legendrian boundary condition. Besides this open-string case, there is a
nice way of dealing with this obstacle by \emph{quantizing} the charge class
by considering the canonical \emph{$\frak{lcs}$-fication}
$$
(Q \times S^1, d\lambda + d\theta \wedge \lambda)
$$
of the contact manifold $(M,\lambda)$ and consider the equation
$$
\delbar^\pi w = 0, \quad w^*\lambda \circ j = g^*d\theta
$$
for a map $u = (w,g): \dot \Sigma \to M \times S^1$
where $\frac{1}{2\pi} [d\theta]$ is the standard generator of $H^1(S^1,\Z)$.
Then we decompose the moduli spaces
into sub-moduli space and handling them separately according to the \emph{charge class},
the cohomology class of the closed one-form $w^*\lambda \circ j$. This enable us to construct
a compactfication of  the moduli space of instantons as carried out in \cite{oh-savelyev}.
We anticipate that this construction will be useful for the study of topology of the group
 $\Cont(M,\xi)$ of contactomorphisms. (See \cite[Section 2.2]{oh-savelyev}.)

\subsubsection{Hamiltonian perturbed contact instantons and contact dynamics}

The genuine power of our intrinsic framework without taking the symplectization
lies in the study of \emph{perturbed contact instantons
\be\label{eq:perturbed-contacton}
(dw - X_H \otimes \gamma)^{\pi(0,1)} = 0, \quad
d\left(e^{g_t(w)}(w^*\lambda + H \otimes \gamma)\circ j\right) = 0
\ee
with Legendrian boundary condition} and its application to \emph{contact Hamiltonian dynamics}
\cite{oh:entanglement1}, and also to construction of
the relevant contact spectral invariants \cite{oh-yso:spectral}.
We recall the correct definition of action functional given in
\cite{oh:entanglement1,oh:contacton-gluing} in this regard.

In \cite[Appendix A]{oh:entanglement1},
it is shown that this Hamiltonian-perturbed contact instanton equation is
\emph{not} the projection of the standard Hamiltonian-perturbed Floer
trajectory equation of the homogeneous lifting $\widetilde H$ on the symplectization $SM$
of the contact Hamiltonian $H$ on $M$. While they coincide when $H = 0$ (under the charge
vanishing), \eqref{eq:perturbed-contacton}  is not the reduction of the
Floer trajectory equation on the symplectization:
The latter equation does not interact well with the contact Hamiltonian calculus
in generating the optimal energy estimates e.g., in relation to Sandon-Shelukhin
type conjecture: It is an open question
whether or not the existing technology in the literature such
as Rabinowitz-Floer homology and
others can reproduce the \emph{optimal result} proved in \cite[Theorem 1.13]{oh:entanglement1}
\emph{on arbitrary compact contact manifolds}. (See
\cite[Example 1.4 \& Lemma 1.6]{rizell-sullivan} to see that the results established in \cite[Theorem 1.13]{oh:entanglement1}
is optimal on $\R^5$ equipped with the standard contact structure.)
While contact dynamics can be lifted to the \emph{homogeneous Hamiltonian dynamics} in the symplectization,
 this lifting operation is \emph{not reversible}, at least easily. In this regard, we strongly believe that the
geometro-analytical framework of perturbed contact instantons is more flexible and
convenient framework than the existing frameworks on the symplectization or of
the Rabinowitz-Floer homology at least
in the study of contact Hamiltonian dynamics and relevant spectral invariants.
We would like to compare how \cite[Theorem 1.13]{oh:entanglement1} is
proved for arbitrary (tame) contact manifolds in Part II thereof
with other existing works  related to the Sandon-Shelukhin's conjecture
\cite{albers-merry,albers-fuchs-merry}, \cite{rizell-sullivan:capacity}: While
the former provides a precise optimal estimate, the latters do not.
The kind of optimal study given in \cite{oh:entanglement1}, via the construction of
spectral invariants, will be further
investigated in \cite{oh-yso:spectral} in the one-jet bundle, and in \cite{oh:entanglement2}
in a more categorical point of view utilizing the Fukaya-type construction and
relevant filtration structure.

\subsubsection{Construction of continuation induced chain map}

The aforementioned optimal estimate partially comes from the way how we
define a continuation-induced map in the framework of contact instanton Floer homology
constructed in \cite{oh:entanglement1,oh:contacton-gluing} whose explanation is
now in order.

Another advantage of our framework over the symplectization lies in the construction of
continuation map induced by Hamiltonian perturbations.
As mentioned in  \cite[Remark 13.2]{oh:contacton-gluing}, the invariance proof
of the relevant Floer homology of \emph{compact} Legendrian submanifolds is rather subtle
in the context of \emph{Legendrian contact homology constructed via the symplectization} in the literature of
contact topology because of the issue of `\emph{moving the infinity of the cylindrical Lagrangian}'.
The common invariance proof of contact homology in the literature uses the argument using the exact symplectic cobordism
and is in the spirit rather different from our invariance proof given in \cite[Section 14]{oh:contacton-gluing}
(See \cite{EGH} for a sketch of such a proof. This was carried out in \cite[Appendix B]{ekholm-rational}
under some technical assumptions on the contact manifold.)
On the other hand the invariance proof given in \cite[Section 14]{oh:contacton-gluing}
is rather straightforward which is similar
to that of Floer cohomology of \emph{compact} Lagrangian submanifolds given in
\cite{oh:cpam1} which uses the \emph{moving boundary condition}.

\bigskip

\noindent{\bf Acknowledgement:}
We thank Rui Wang for the collaboration of the works \cite{oh-wang1,oh-wang2,oh-wang3}
which the current research is partially based on. We also thank her for some useful
comments on the present paper.

\section{Three elliptic twistings of contact Cauchy-Riemann map equation}
\label{sec:elliptictwisting}

The contact Cauchy-Riemann equation itself $\delbar^\pi w = 0$ does not form
an elliptic system because it is degenerate along the Reeb direction: Note that the rank of
$w^*TQ$ has $2n+1$ while that of $w^*\xi\otimes \Lambda^{0,1}(\Sigma)$ is $2n$.
Therefore to develop suitable deformation theory and a priori estimates, one needs to
lift the equation to an elliptic system. In hindsight, the pseudoholomorphic curve
system of the pair $(w,a)$ is one such lifting via introducing an auxiliary variable $a$ a function on the
Riemann surface, when the one-form $w^*\lambda \circ j$ is exact.
Hofer \cite{hofer} did this by lifting the equation to the symplectization $Q \times \R$
and considering the pull-back function $a: = s\circ w$ of the $\R$-coordinate function
$s$ of $Q \times \R$. By doing so, he \emph{added one more variable} to the equation $\delbar^\pi w =0$
while \emph{adding 2 more equations $w^*\lambda \circ j = da$} and produced an elliptic
system which is exactly becomes Gromov's pseudoholomorphic curve system on the symplectization
$Q \times \R$.

\subsection{Contact instanton lifting of contact Cauchy-Riemann map}

It turns out, again by hindsight, the current contact instanton map system
\be\label{eq:contact-instanton}
\delbar^\pi w = 0, \quad d(w^*\lambda \circ j) = 0
\ee
is such an elliptic lifting which is more natural in some respect in that it does not introduce
any additional variable and keeps the original `bulk', the contact manifold $Q$.

The relevant a priori (local) elliptic estimates and the global exponential decay
estimates near the puncture of the punctured Riemann surface $\dot \Sigma$
have been established in \cite{oh-wang2}.
This is the lifting whose study is the main theme of the present paper and
is also closely related to the following lifting
of \emph{gauged sigma model with abelian Hick's field}.
We would like to emphasize that this lifting includes the study of
pseudoholomorphic curves in symplectization as the special case of
exact $w^*\lambda \circ j$.

\subsection{Gauged sigma model lifting of contact Cauchy-Riemann map}

There is another lifting of $w$ this time involving a section of complex line bundle
\be\label{eq:CLlambda}
\CL_\lambda \to Q
\ee
whose fiber at $q \in Q$ is given by
$$
\CL_{\lambda,q} = \R_{\lambda,q} \otimes \C
$$
where $\R_\lambda \to Q$ is the trivial real line bundle whose fiber at $q$ is given by
$$
\R_{\lambda,q} = \R \{X_\lambda(q)\}.
$$
Note that $\CL_q$ has a canonical identification with the bundle
$$Q \times \R_+
T(Q \times \R_+)|_{\{r =1\}} =   \xi \oplus \R \cdot X_\lambda \oplus \R \cdot \frac{\del}{\del s}
$$
in the symplectization $Q \times \R_+$.

Now let $w: \Sigma \to Q$ be a smooth map where $\Sigma$ is either closed or
a punctured Riemann surface, and $\chi$ be a section of the pull-back bundle
$w^*\CL_\lambda$.

\begin{defn} We call a triple $(w,j,\chi)$ consisting of a complex structure $j$ on $\Sigma$,
$w: \Sigma \to Q$ and a $\C$-valued one-form
$\chi$ a \emph{gauged contact instanton} if they satisfy
\be\label{eq:gauged-instanton}
\begin{cases}
\delbar^\pi w = 0 \\
\delbar \chi = 0, \quad
\Im \chi = w^*\lambda.
\end{cases}
\ee
\end{defn}

This system is a coupled system of the contact Cauchy-Riemann map equation and
the well-known Riemann-Hilbert problem of the type which solves the real part in terms of
the imaginary part of holomorphic functions in complex variable theory.

\subsection{Pseudoholomorphic lifting of contact Cauchy-Riemann map}

The above two liftings do not have any restriction on the cohomology class $[w^*\lambda \circ j] \in H^1(\dot \Sigma;\R)$.
On the other hand, there is the more commonly known elliptic twisting \emph{under the restriction
that $w^*\lambda \circ j$ is exact}, and \emph{with the specification of the
anti-derivative of $w^*\lambda \circ j$} as an auxiliary variable $a: \Sigma \to \R$ by requiring
$$
w^*\lambda \circ j = da
$$
whose expounding is now in order. We call a contact instanton \emph{exact} if $[w^*\lambda \circ j] = 0$.

\begin{rem}\label{rem:exact}
We would like to point out that the exact case itself
forms a closed realm in the study of contact instantons and does not need to involve symplectization in its study.
If we restrict to the exact contact instantons, any adiabatic puncture with $T = 0$ will
be removable as in the case of pseudoholomorphic curves by Theorem \ref{thm:removable}. This enables us to perform the standard
Gromov-Floer theory type compactification of the moduli space of exact contact
instantons and to define a Floer homology type invariants. However the geometry of contact instantons
is not exactly the same as that of pseudoholomorphic curves in symplectization and so we do not
expect the algebraic structures of the contact homology type invariants coincide.
In \cite{oh-savelyev}, we construct a compactification of the moduli space of contact
instantons \emph{with each fixed charge class} which includes the exact case
(e.g., the case of pseudoholomorphic curves in symplectization) as a subcase, and
give a construction of the relevant contact homology type invariants in
\cite{oh:entanglement1,oh:contacton-gluing}.
\end{rem}

We consider the canonical symplectization $E \to Q$ (see
section \ref{sec:canonical}).
Note that in the presence of contact form $\lambda$, any smooth map
$w:\dot \Sigma \to Q$ can be naturally lifted to a map $\widetilde w :\dot \Sigma \to W$
so that
\be\label{eq:tildew}
\widetilde w(z) = a(z) \lambda(w(z)) \in W_{w(z)} \subset T^*_{w(z)}Q
\ee
for some function $a: \dot \Sigma \to \R_+$ or equivalently to a map
$$
(w,a) : \dot \Sigma \to Q \times \R
$$
via the trivialization $\exp\circ \Phi_\lambda: Q \times \R \to W$.

Now we equip $(Q,\xi)$ with a triad $(Q,\lambda, J)$ and the cylindrical
almost complex structure $\widetilde J = J_0 \oplus J$. Then
the derivative $d\widetilde w = dw \oplus da \frac{\del}{\del s}$ can be further decomposed to
\be\label{eq:dtildew}
d\widetilde w(z) =  d^\pi w \oplus w^*\lambda X_\lambda \oplus da\frac{\del}{\del s}.
\ee
as a $TW$-valued 1-form with respect to the splitting
$$
\operatorname{Hom}(T_z\dot \Sigma, T_{\widetilde w(z)}W)
=  \operatorname{Hom}(T_z\dot \Sigma, HT_{\widetilde w(z)}W)
\oplus \operatorname{Hom}(T_z\dot \Sigma, VT_{\widetilde w(z)}W).
$$
By definition, we have
$$
d\pi  \widetilde{dw} = dw.
$$
It was derived by Hofer \cite{hofer} that $\widetilde w$ is $\widetilde J$-holomorphic
if and only if $(w,a)$ satisfies
\be\label{eq:tildeJ-holo}
\begin{cases}
\delbar^\pi w = 0 \\
w^*\lambda \circ j= da.
\end{cases}
\ee

\section{Canonical symplectization and Hofer's $\lambda$-energy; revisit}
\label{sec:canonical}

In this subsection, we first recall the canonical symplectization of contact manifold $(Q,\xi)$
explained in Appendix 4 \cite{arnold-book}, which does not involve the choice of contact form.
We denote this canonical symplectization by $(W,\omega_W)$ which is defined to be
\be\label{eq:setalpha}
\{\alpha \in T^*Q \mid \alpha \neq 0, \, \ker \alpha = \xi \} \subset T^*Q \setminus \{0\}.
\ee
When $Q$ is oriented and a positive contact form $\lambda$ is given, we can canonically lift
a map $w: \dot \Sigma \to Q$ to a map $\widehat w:\dot \Sigma \to W$.
We then examine the relationship between $w$ being a contact instanton and $\widehat w$
being a pseudoholomorphic curves on $W$ with respect to scale-invariant almost complex
structure on $W$. We give a geometric description of Hofer's remarkable energy
introduced in \cite{hofer} in terms of this canonical symplectization. This energy
is the key ingredient needed in the bubbling-off analysis and so in the construction of the
compactification of the moduli spaces of pseudoholomorphic curves needed to develop the
symplectic field theory \cite{EGH}, \cite{behwz}. In section \ref{sec:offshellenergy},
we will then introduce its variant for the study of contact instanton maps whose
charge is not necessarily vanishing, i.e. $w^*\lambda \circ j$ does not have to be exact.

Consider the $(2n+2)$-dimensional submanifold $W$ of $T^*Q$ defined in \eqref{eq:setalpha}.
When we fix an orientation $Q$, we can consider
\be\label{eq:W}
W = \{\alpha \in T^*Q \setminus \{0\} \mid \ker \alpha = \xi, \, \alpha(\vec n) > 0\}
\ee
where $\vec n$ is a vector such that $\R \{\vec n\} \oplus \xi$ becomes a positively
oriented basis. Note that $W$ is a principal $\R_+$-bundle over $Q$ that is trivial.

We denote by $i_W: W \hookrightarrow T^*Q$ and by $\Theta$ the Liouville
one-form on $T^*Q$. The basic proposition is that $W$ carries the canonical symplectic form
$$
\omega_W = -i_W^*d\Theta.
$$
One important point of  this canonical symplectization is the fact that it depends only on the orientation of $Q$
but does not depend on the choice of contact form $\lambda$. The symplectic form $\omega_W$
provides a natural symplectic (Ehresmann) connection provided by the splitting
\be\label{eq:splitting1}
TW = \widetilde{TQ} \oplus VTW
\ee
where $VTW$ is the vertical tangent bundle and
\be\label{eq:omegaWperp}
\widetilde{TQ}|_\alpha = \{\eta \in T_\alpha W \mid \omega_W(\eta, \cdot) \equiv 0 \}.
\ee
\par
Now we choose a contact form $\lambda$ so that $\lambda \wedge (d\lambda)^n$ is
positive with respect to the given orientation. Since $\lambda$ provides a section of
of $W \to Q$, it induces a trivialization of $W$ as the principal $\R_+$-bundle
$$
\Phi_\lambda: Q \times \R_+ \to W; \quad (r,q) \mapsto r\, \lambda(q)
$$
which in turn leads to the natural isomorphism
\be\label{eq:splitting2}
 TQ \oplus \R \cong TW =  \widetilde{TQ} \oplus \R \cdot \lambda
 \ee
defined by $(Z,c) \mapsto \widetilde Z \oplus c \lambda$.
Combining this with \eqref{eq:splitting1},
we obtain the splitting
\be\label{eq:omegaWperp2}
TW = \widetilde \xi \oplus  \widetilde{\R\cdot X_\lambda} \oplus \widetilde {\R \cdot \lambda}.
\ee
We note that there is a canonical paring on
$ \widetilde{\R\cdot X_\lambda} \oplus \widetilde {\R \cdot \lambda}$
given by
$$
\langle \lambda, X_\lambda \rangle = 1
$$
and so it carries the canonical symplectic form thereon. We summarize the above discussion into
\begin{prop}\label{prop:canonical-symp}
Suppose $Q$ is given an orientation and a positive contact form $\lambda$. Then it provides a
natural $\R_+$-equivariant symplectomorphism
$$
\Phi_\lambda: Q \times \R_+ \to W
$$
whose derivative induces a canonical $\R_+$-equivariant symplectic vector bundle isomorphism
$$
d\Phi:(TQ \oplus \R^2, d\lambda \oplus \omega_{0,2}) \to (TW, \omega_W)
$$
\be\label{eq:dPhilambda}
d\Phi_\lambda(Z, b, a) =  \widetilde Z +a\, \widetilde \lambda + b\, \widetilde X_\lambda.
\ee
\end{prop}
The usual symplectization of $(Q,\lambda)$ used in the literature is nothing but $\R_- \times Q$ with the pull-back
symplectic form $(\Phi_\lambda)^*\omega_W$ thereto, which can be explicitly written as
$$
(\Phi_\lambda)^*\omega_W = (\Phi_\lambda \circ i_W)_*\Theta = d(r\, \pi^*\lambda)
$$
where $r = r_\lambda \in \R^+$ is the radial coordinate such that the embedding $Q \hookrightarrow W$
corresponds to the hypersurface $r=1$ and $\pi: Q \times \R_+ \to Q$ the
projection. If we now pull-back this form to $Q \times \R$ by the diffeomorphism
$\exp: Q \times \R \to Q \times \R_+$ defined by $\exp(s,q) = (e^s,q)$, then the corresponding symplectic form becomes
$$
e^s(\pi^*d\lambda + ds \wedge \pi^*\lambda), \quad \pi: Q \times \R \to Q.
$$
\par
Next we involve an endomorphism $J: \xi \to \xi$ with $J^2 = -id$ such that $(\xi,J,g_\lambda)$
with $g_\lambda = d\lambda(\cdot, J \cdot)$ becomes a Hermitian vector bundle. For the purpose of
doing analysis on $Q \times \R$, we need to provide a \emph{cylindrical metric} thereon which we
choose
$$
g_\lambda + dr^2 = d\lambda(\cdot, J\cdot) + \lambda \otimes \lambda + dr^2
$$
and cylindrical almost complex structure
$$
\tilde J = J_0 \oplus J
$$
on $T(Q \times \R) \cong \R\{\frac{\del}{\del r}\} \oplus \R\{X_\lambda \} \oplus \xi$.
On the other hand, the pull-back symplectic form becomes
$$
e^s(\pi^*d\lambda + ds \wedge \pi^*\lambda)
$$
which is \emph{not cylindrical}. The above fact that the pull-back symplectic form is not
cylindrical makes the topological control of
the \emph{full} harmonic energy of a $\widetilde J$-holomorphic map $u: \Sigma \to Q \times \R$
by the symplectic area of this symplectic form not possible in general, \emph{unless one has the
control of the coordinate $a = s \circ w$}.

Instead one tries to control the \emph{local} (in target)
harmonic energy by considering the map
$$
\widehat \psi: Q \times \R \to W; \quad \widehat \psi(s,x) = \psi(s) \, (\pi^*\lambda)(x)
$$
associated to each monotonically increasing function $\psi$ such that
\be\label{eq:varphi}
\psi(s) =
\begin{cases} 1 \quad & \mbox{for $s \geq R_1$} \\
\frac{1}{2} \quad & \mbox{for $s \leq R_0$}
\end{cases}
\ee
for any pair $R_0 < R_1$ of real numbers. We measure the symplectic area
of the composition $\widehat\psi \circ w: \dot \Sigma \to W$
for all possible variations of such $\psi$.  Hofer's original definition of
this type of energy then can be expressed as the integral
\bea
E_\CC(u)& : = & \sup_{\psi} \int_{\dot \Sigma} (\widehat\psi \circ u)^*\omega_W \label{eq:hofer-energy2}\\
& = & \sup_{\psi} \int_{\dot \Sigma} (\widehat\psi \circ u)^*d (r\, \pi^*\lambda) \nonumber\\
& = & \sup_{\psi} \int_{\dot \Sigma} d (\psi(s)\, \pi^*\lambda) \label{eq:hofer-energy}\\
& = & \sup_{\psi}\left(\int_{\dot \Sigma} \psi(a) dw^*\lambda + \psi'(a)\, da \wedge w^*\lambda\right).
\eea
Note that \eqref{eq:hofer-energy} is precisely the same as Hofer's original definition of his
energy given in \cite{hofer}. Later in \cite{behwz}, the authors split this energy into two parts,
one purely depending on $w$
$$
E^\pi(w)= \int_{\dot \Sigma} dw^*\lambda
$$
and the other
$$
E^\lambda(u) =  \sup_{\psi} \int_{\dot \Sigma} \psi'(a)\, da \wedge w^*\lambda.
$$
In retrospect, it was an amazing insight of Hofer \cite{hofer} that this way of considering nicely controls the
bubbling-off analysis when there is no apparent way of controlling the asymptotic behaviour of the
bubble map $\C \to Q \times \R$ when the bubble map is not confined in a compact domain of
$Q \times \R$.

\section{Jenkins-Strebel quadratic differential and minimal area metrics}
\label{sec:jenkins}

For any given marked Riemann surface $(\Sigma, \{r_1,\cdots, r_k\})$, we denote by
$\dot \Sigma$ the associated punctured Riemann surface. We assume either $\genus \Sigma \geq 1$ or
$\genus \Sigma = 0$ with $k \geq 2$.

Following Zwiebach \cite{zwiebach}, we give a description of the notion of minimal area metric
associated to the given punctured Riemann surface $\dot \Sigma$ and its relationship with
the Jenkins-Strebel quadratic differentials. We also refer to section 2 of Bergmann's preprint \cite{bergmann} for
some discussion that is in the similar spirit as that of this section.

\begin{defn} A metric $h = \rho \, |dz|$ is called \emph{admissible} for a set of
constants $A_j$ if
$$
\int_{\gamma; \gamma \sim \gamma_j} \rho\, |dz| \geq A_j
$$
for any curve $\gamma$ homotopic to $\gamma_j$ in $\dot \Sigma$.
\end{defn}

In this metric, one has the semi-infinite tubes of circumference $\ell \geq A_j$ at each
puncture $r_j$. Near the puncture $r_j$, one must have
$$
\rho^2(z) \sim (A_j/2\pi |z|)^2.
$$
\begin{defn}[Reduced area \cite{zwiebach}] The reduced area, denoted by $\Area^{red}(\Sigma,h)$ is given by
\be\label{eq:reducedarea}
\Area^{red}(\Sigma,h) = \lim_{\delta \to 0}
\left(\int\int_{\Sigma(\delta)} dA + \frac{1}{2\pi} \ln \delta \sum_{j=1}^k A_j^2\right)
\ee
where $\Sigma(\delta)$ denotes the surface obtained by excising the discs $|z_j| \leq \delta$ from $\Sigma$.
\end{defn}

\begin{defn}[Minimal area metric] A metric $h$ on $\dot \Sigma$ is called a \emph{minimal area metric}
if the reduced area is minimal among all possible metrics arising from quadratic differentials.
\end{defn}

From now on, we restrict ourselves to the case of $g=0$.
We will need the following basic existence and the uniqueness result proved in \cite{zwiebach}

\begin{thm}[Zwiebach \cite{zwiebach}] When $g = 0$ and $k \geq 3$, there exists a unique
minimal area metric associated to each $(\Sigma, j) \in \CM_{0,k}$, which continuously extends to
the compactification $\overline \CM_{0,k}$.
\end{thm}
In other words, the minimal area metric provides a natural slice to the well-known
isomorphism between the set of complex structures and the set of conformal isomorphism classes of
associated metrics, which respects the sewing rule of the degeneration of conformal structures.
A similar representation of the conformal structure on the boundary punctured discs,
the open string analogue of the above theorem, was
used in Fukaya and the author's work \cite{f-oh} in their study of
adiabatic degeneration of pseudo-holomorphic polygons with Lagrangian
boundaries on the cotangent bundle.

It is also shown that each minimal area metric arises from Jenkins-Strebel quadratic differential
\cite{jenkins}, \cite{strebel} whose singularities are at most a pole. Some brief account on
Jenkin-Strebel quadratic differential should be in order.
A quadratic differential $\varphi$ on a Riemann surface
$\Sigma$ is a set of function elements $\phi_i(z_i)$, meromorphic in the local
coordinates $z_i = x_i + i y_i$ with transformation property
\be\label{eq:quadratic}
\phi_i(z_i)(dz_i)^2 = \phi_j(z_j) (dz_j)^2,
\ee
under a change of local coordinates.
A quadratic differential defines a metric $|\phi_i(z_i)| |dz_i|^2$.

A horizontal trajectory of a quadratic differential is a curve along which $\phi(z) (dz)^2$ is real
and positive.

\begin{defn} A Jenkins-Strebel quadratic differential is a quadratic differential for which the
nonclosed trajectories cover a set of measure zero on the surface.
\end{defn}

A JS quadratic differential decomposes a surface into characteristic ring domains, the maximal
ring domains swept by the closed trajectories.
These ring domains can be annuli or punctured discs.

On a punctured discs $D(1) \setminus \{0\}$ with coordinates $w$, the
JS quadratic differential is given by the form
\be\label{eq:JSquad-w}
\phi_{JS}(z)\, dz^2, \quad  \phi_{JS}(z) = - \frac{a^2}{(2\pi)^2} \frac{1}{z^2}
\ee
where $2\pi a$ is the length of the horizontal trajectory of the associated minimal area metric.
The metric is flat and isometric to the semi-infinite tube $(-\infty,0] \times S^1$ with
coordinates $(\tau,t)$ with $u = \tau + i\, t$ and proportional to the standard metric $d\tau^2 + dt^2$.
This is nothing but the canonical isothermal coordinate of the metric and satisfies
\be\label{eq:dz2}
du^2 = - \frac{a^2}{(2\pi)^2} \frac{1}{z^2} dz^2.
\ee

Under the minimal area metric for the case of $g =0$, $\dot \Sigma$ is a finite union of
$k$ semi-infinite cylinders and a finite set of cylinders with finite height of circumference
$2\pi$. So each cylinder is isometric to the standard
cylinders, either $[0, \infty) \times S^1$ or $[0, \ell] \times S^1$ with metric
$$
h = \left(\frac{a}{2\pi}\right)^2(d\tau^2 + dt^2)
$$
where $(\tau,t)$ is the standard coordinates on the cylinder. On each cylinder
it carries the vector field $\frac{\del}{\del \tau}, \, \frac{\del}{\del t}$ which are invariant under the
transformation
$$
(\tau,t) \mapsto (\tau + \tau_0, t+ t_0)
$$
and so depends only on the metric.  Denote by $S \subset \dot\Sigma$ the
union of sewing seams of the set of cylinders given above. Then $\dot \Sigma$
carries a vector field $V = V(j)$ that restrict to the coordinate vector
field $\frac{\del}{\del \tau}$ on each cylinder. As a result $V(j)$ is discontinuous along the
$S$ but its flow lines form a foliation those leaves are continuous even across the seams.
The vector field $V$ is called the vertical vector field and
the associated foliation is called the vertical foliation of the quadratic differential associated to
the minimal area metric \cite{strebel}. Similarly the vector field $\frac{\del}{\del t}$ glues to define
a global vector field $H(j)$ called the horizontal vector field, which is continuous except at
a finite number of points.

We would like to mention that when we give a distinguished
marked point $r_0$ as the `output' and put the rest as the `input' marked points
as in the definition of $A_\infty$-structures as in \cite{f-oh,fooo:book},
the flow of the vector field $V(j)$ becomes an oriented foliation whose
leaves consist of the flow lines of $V(j)$. Then the flow become continuous even across the
seams.

We will also need to consider the case $g= 0$ and $k = 2$. (See \cite{wolf-zwieb} for
the relevant discussion.) In this case,
$\dot \Sigma$ with the minimal area metric is isometric to the standard cylinder
$\R \times S^1$ with the metric $d\tau^2 + dt^2$. While the metric is uniquely determined, its
associated flat coordinates are defined uniquely modulo the translations and rotations
$$
(\tau,t) \mapsto (\tau+\tau_0,t+t_0), \quad \tau_0 \in \R, \, t_0 \in S^1.
$$

\section{Off-shell energy of contact instantons}
\label{sec:offshellenergy}

Fix a K\"ahler metric $h$ on $(\Sigma, j)$.  The norm $|dw|$ of the map
$$
dw:(T\Sigma,h) \to (TQ, g)
$$
with respect to the metric $g$ is defined by
$$
|dw|_g^2 := \sum_{i=1}^{2} {|dw(e_i)|_g}^2,
$$
where  $\{ e_1, e_2 \}$ is an orthonormal frame of $T \Sigma$
with respect to $h$.

The following are the consequences from the definition of
contact Cauchy-Riemann map and the compatibility of $J$ to $d\lambda$ on $\xi$, whose
proofs we omit but refer to \cite{oh-wang1}.

\begin{prop}\label{prop:energy-omegaarea}
Denote $g_J=\omega(\cdot, J \cdot)$ and the associated norm
by $|\cdot| = |\cdot|_J$. Fix a Hermitian metric $h$ of $(\Sigma,j)$,
and consider a smooth  map $u:\Sigma \to M$. Then we have
\begin{enumerate}
\item[(1)] $|d^\pi w|^2 = |\del^\pi w| ^2 + |\delbar^\pi w|^2$,
\item[(2)] $2w^*d\lambda = (-|\delbar^\pi w|^2 + |\del^\pi w|^2) \,dA $
where $dA$ is the area form of the metric $h$ on $\Sigma$.
\item[(3)] $w^*\lambda \wedge w^*\lambda \circ j = |w^*\lambda|^2\, dA$
\item[(4)] $|\nabla w^*\lambda|^2 = |dw^*\lambda|^2 + |\delta w^*\lambda|^2$.
\end{enumerate}
\end{prop}

We then introduce the $\xi$-component of the harmonic energy, which we call
the $\pi$-harmonic energy. This energy equals the contact area $\int w^*d\lambda$
`on shell' i.e., for any contact Cauchy-Riemann map, which satisfies $\delbar^\pi w = 0$

\begin{defn}\label{defn:pi-energy}
For a smooth map $\dot \Sigma \to Q$, we define the $\pi$-energy of $w$ by
\be\label{eq:Epi}
E^\pi(j,w) = \frac{1}{2} \int_{\dot \Sigma} |d^\pi w|^2.
\ee
\end{defn}

As discovered by Hofer in \cite{hofer} in the context of symplectization, the $\pi$-harmonic energy
itself is not enough for the crucial bubbling-off analysis needed for the equation
\eqref{eq:contact-instanton}. This is only because the bubbling-off analysis
requires the study of asymptotic behavior of contact instantons on the complex place $\C$.
A crucial difference between the current case of contact instantons from Gromov's theory of
pseudoholomorphic curves on symplectic manifolds is that there is no removal singularity
result of the type of harmonic maps (or pseudoholomorphic maps). Because of this, one
needs to examine the $X_\lambda$-part of energy that controls the asymptotic behavior of
contact instantons near the puncture. For this purpose, the Hofer-type energy
introduced in \cite{hofer} is crucial. In this section, we generalize this energy
to the general context of non-exact case without involving the symplectization.

Following the modification made in \cite{behwz} of Hofer's original definition \cite{hofer} (and denoting
$\varphi = \psi'$ for the function $\psi$ given in section \ref{sec:canonical}), we introduce the following class of
test functions
\begin{defn}\label{defn:CC} We define
\be
\CC =\{\varphi: \R \to \R_{\geq 0} \mid \supp \varphi \, \text{is compact}, \, \int_\R \varphi = 1\}.
\ee
\end{defn}

Let $w:\dot \Sigma \to Q$ be a contact instanton with the asymptotic
charge $Q(p)$ at the puncture. Recall this
number depends only on the homology class $[\gamma]$ of the loop
$\gamma = w|_{D_\delta(p)}(\tau,\cdot) \subset \dot \Sigma \setminus \{p\}$
by the closedness equation of $w^*\lambda\circ j$, which does not depend on $\tau$ either.

Then on the given cylindrical neighborhood $D_\delta(p) \setminus \{p\}$, we can write
$$
w^*\lambda \circ j + Q(p)\, dt = df
$$
for some function $f: [0,\infty) \times S^1 \to \R$. Here $dt$ is the one-form that is made of
the one-form $dt$ defined before on each cylinder. The form is globally continuous except at the finite
number points at which the vector field $\frac{\del}{\del t}$ is not continuous.
We call $f$ the \emph{contact instanton potential}.

We remark that when $w$ is given, the function $f$ on $D_\delta(p) \setminus \{p\}$ is uniquely determined modulo
the shift by a constant.

\begin{defn}[$E_\CC$-energy] Let $w$ satisfy $d(w^*\lambda \circ j) = 0$. Then we define
$$
E_{\CC}(j,w) = \sup_{\varphi \in \CC} \int_\Sigma d(\psi(f)) \wedge df\circ j
= \sup_{\varphi \in \CC} \int_\Sigma d(\psi(f)) \wedge (- w^*\lambda + Q(p)\, d\tau).
$$
\end{defn}
 We note that
$$
d(\psi(f)) \wedge df \circ j = \psi'(f) df \wedge df\circ j= \varphi(f) df \wedge df\circ j \geq 0
$$
and hence we can rewrite $E_{\CC}(j,w)$ into
$$
E_{\CC}(j,w) = \sup_{\varphi \in \CC} \int_\Sigma \varphi(f) df \wedge df\circ j.
$$
\begin{prop}\label{prop:a-independent} For a given smooth map $w$ satisfying $d(w^*\lambda \circ j) = 0$,
we have $E_{\CC;f}(w) = E_{\CC,g}(w)$ whenever $df = w^*\lambda\circ j + Q(p)\, dt = dg$
on $D^2_\delta(p) \setminus \{p\}$ (and so $g(z) = f(z) + c$ for some constant $c$ on each connected component of
$Q$).
\end{prop}
\begin{proof}
Certainly $df$ or $df\circ j$ are independent of the addition by constant $c$.
On the other hand, we have
$$
\varphi(g) = \varphi(f + c)
$$
and the function $a \mapsto \varphi(a + c)$ still lie in $\CC$. Therefore after taking
the supremum over $\CC$, we have derived
$$
E_{\CC,f}(j,w) = E_{\CC,g}(j,w).
$$
This finishes the proof.
\end{proof}

This proposition enables us to introduce the following

\begin{defn}[$\lambda$-energy at a puncture $p$]
We denote the common value of $E_{\CC,f}(j,w)$ by $E^\lambda_p(w)$, and call the \emph{$\lambda$-energy at $p$}.
\end{defn}

The following then would be the preliminary definition of the total energy.

\begin{defn}[Total energy] Let $w:\dot \Sigma \to Q$ be any smooth map.
We define the total energy of $w$ by
\be\label{eq:total-energy-sum}
E(j,w) = E^\pi(j,w) + \sum_{l=1}^k E^\lambda_{p_l}(j,w).
\ee
We denote
$$
E^\lambda(j,w) = \sum_{l=1}^k E^\lambda_{p_l}(j,w).
$$
\end{defn}

\begin{rem} \begin{enumerate}
\item To take further analogy with physics, one may regard the $\pi$-harmonic energy
as the `kinetic energy' of the contact instanton and the $\lambda$-energy as
the `potential energy' thereof respectively.
\item The above definition is unsatisfying and incomplete as an off-shell energy of the pair
$(j,w)$ \emph{when we vary complex structure $j$} on the punctured surface $\dot \Sigma$.
For this purpose, we need to involve the complex structure in the definition of $E^\lambda$
also like $E^\pi(j,w)$ does. This is where the Zwiebach's notion of minimal area metric
\cite{zwiebach}, \cite{wolf-zwieb} enters which extends the
cylindrical structure to the full Riemann surface not just to the punctured neighborhoods.
\end{enumerate}
\end{rem}

In the rest of the section, we assume $\Sigma$ has genus 0. The reason for this restriction is
only because for the higher genus case, the minimal area metric representation of conformal
structure is over-counting \cite{zwiebach}. Other than this, the discussion below
is equally applied to any conformal structure represented by a minimal area metric.

First, we assume $k \geq 2$, i.e., the number of marked
points at least 2.  In this case, the conformal structure carries the
minimal area metric representation \cite{zwiebach}. Under the minimal area metric for the case of $g =0$,
$\dot \Sigma$ is a finite union of
$k$ semi-infinite cylinders and a finite set of cylinders with finite height of circumference
$2\pi$. So each cylinder is isometric to the standard
cylinders, either $[0, \infty) \times S^1$ or $[0, \ell] \times S^1$ with metric
$$
h = \left(\frac{a}{2\pi}\right)^2(d\tau^2 + dt^2)
$$
where $(\tau,t)$ is the standard coordinates on the cylinder. On each cylinder
it carries the vector field $\frac{\del}{\del \tau}$ which is invariant under the
transformation
$$
(\tau,t) \mapsto (\tau + \tau_0, t+ t_0)
$$
and so depends only on the metric. Denote by $S \subset \dot\Sigma$ the
union of sewing seams of the set of cylinders given above.
We label the marked points as $\{r_0,\cdots, r_k\}$ for $k \geq 1$ so
that $r_0$ is incoming and the rest are outgoing. Then $\dot \Sigma$
carries a vector field $V = V(j)$ that is rotationally invariant and restricts to the coordinate vector
field $\frac{\del}{\del \tau}$ on each cylinder. (Here `$V$' stands for `vertical' since
the meridian circles are often called `horizontal foliation'.) (See section \ref{sec:jenkins}
and \cite{jenkins,strebel,zwiebach}.)

We can also associate a tree $T$ consisting of the cores of the above cylinders
that is naturally oriented consistently with the unique incoming assignment of
the puncture $r_0$. We denote by $\ell(e)$ the length of the edge $e$ of the
tree. There is also the unique incoming exterior edge incident to $r_0$ and
the unique interior vertex of the exterior edge. We denote by $v^{dist}$ the
unique distinguished interior vertex.

Denote by $Q(r_i) = Q(e_i^{ext})$ the charge at the puncture $r_i$, and assign these numbers to
the exterior edges incident to the punctures respectively. We then associate
charge $Q(e)$ to each interior edge $e$ so that the following balancing
condition holds
\be\label{eq:balanceatv}
\sum_{e \in E(v)} Q(e) = 0
\ee
for all interior vertex $v \in V^{int}(T)$ where $E(v)$ is the set of edges
incident to the vertex $v$. This uniquely determines the charge function
$Q: E(T) \to \R$. Furthermore this balancing condition makes the following
lemma hold.

\begin{lem}\label{lem:Qdt} Consider the current $\sum_{e \in E(T)} Q(e)\, dt_e$, i.e.,
the distributional one-form on $\dot \Sigma$. Then it is closed as a current,
provided \eqref{eq:balanceatv} holds at every interior vertex $v \in V(T)$.
\end{lem}

We remark that the coordinate $t_e$ defined up to the rotation of $S^1$
can be uniquely determined by assigning a tangent direction at each puncture. But the one-form
$Q(e)\, dt_e$ is well-defined independently of the rotations.  In particular the current
$$
\sum_{e \in E(T)} Q(e)\, dt_e
$$
is smooth away from a finite number of Lipschitz singularities located in the sewing
seams.

Next we associate the charges $Q(w;e)$ of contact instanton $w$ by the integrals
$$
Q(w;e) = - \int_{S^1_e} w^*\lambda \circ j
$$
where $S^1_e$ is a meridian circle of the cylinder associated to the edge $e \in E(T)$.
Then we consider the one-form
$$
w^*\lambda \circ j + \sum_{e \in E(T)} Q(e)\, dt_e
$$
as a current, where $(\tau_e,t_e) \in [0,\ell(e)] \times S^1$ the natural cylindrical coordinates on the cylinder
associated to the edge $e$. By construction this current is exact and so we can solve the distributional equation
$$
w^*\lambda \circ j + \sum_{e \in E(T)} Q(w;e)\, dt_e = df
$$
a priori for some distribution $f$.

\begin{prop} The distribution $f$ is a continuous function on $\dot \Sigma$ which
is smooth away from the singularities mentioned above.
\end{prop}
\begin{proof}
By the property of the minimal area metric which is rotationally symmetric on each
cylinder, the function $f$ depends only on the coordinate
$\tau_e$ and can be uniquely determined by the integral formula
$$
f(z) = \int_{v^{dist}}^z \left(w^*\lambda \circ j + \sum_{e \in E(T)} Q(w;e)\, dt_e\right)
$$
and setting the normalization condition
\be\label{eq:normalizedf}
f(v^{dist}) = 0.
\ee
This integral is path-independent by the exactness of the current and so is well-defined. All the properties
stated then immediately follows from the expression of $f$.
\end{proof}

This function $f$ seems to deserve a name.

\begin{defn}[Contact instanton potential] We call the above normalized function $f$
the \emph{contact instanton potential} of the contact instanton charge form $w^*\lambda \circ j$.
\end{defn}

If $\dot \Sigma$ carries only one puncture, $\dot \Sigma \cong \C$ and so cannot
carry the above minimal area representation but in this case
the closed form $w^*\lambda \circ j$ is automatically exact.
Therefore there exists a function $f: \dot \Sigma \to \R$ such that $w^*\lambda \circ j = df$
in which case we may regard the pair $(f,w)$ as a pseudoholomorphic map to the
symplectization as in \cite{hofer}.

Now we define the final form of the off-shell energy.
Let $w:\dot \Sigma \to Q$ be any smooth map.
We define the total energy of $w$ by
\be\label{eq:total-energy}
E(j,w) = E^\pi(j,w) + E^\lambda(j,w)
\ee
We define
\be\label{eq:finallambdaenergy}
E^\lambda(j,w) = \sup_{\varphi \in \CC} \int_\Sigma d(\varphi(f)) \wedge df\circ j.
\ee
This energy will be used in our construction of the compactification of moduli space of contact
instantons of genus 0 in a sequel. In the rest of the paper, we suppress $j$ from the arguments of
the energy $E(j,w)$ and just write $E(w)$.

\section{Contact instantons on the plane}
\label{subsec:onC}

As in Hofer's bubbling-off analysis in pseudo-holomorphic curves on symplectization \cite{hofer}, it
turns out that study of contact instantons on the plane plays a crucial role in the bubbling-off
analysis of contact instantons too.

We recall the following useful lemma from \cite{hofer-viterbo} whose proof we refer thereto.

\begin{lem}\label{lem:Hoferslemma} Let $(X,d)$ be a complete metric space, $f: X \to \R$ be a
nonnegative continuous function, $x \in X$ and $\delta > 0$. Then there exists $y \in X$ and
a positive number $\epsilon \leq \delta $ such that
$$
d(x,y) < 2 \delta, \, \max_{B_y(\epsilon)} f \leq 2 f(y), \, \epsilon\, f(y) \geq \delta f(x).
$$
\end{lem}

For this purpose, we start with a proposition which is an analog to
Theorem 31 \cite{hofer}. Our proof is a slight modification and some simplification
of Hofer's proof
of Theorem 31 \cite{hofer} in our generalized context.

\begin{prop}\label{prop:C^1}
Let $w: \C \to Q$ be a solution of \eqref{eq:contact-instanton}. Regard $\infty$ as
a puncture of $\C  = \C P^1\setminus \{\infty\}$. Suppose $|dw|_{C^0} < \infty$ and
\be\label{eq:asymp-densitybound}
E^\pi(w)=0, \quad  E^\lambda_{\infty}(w) < \infty.
\ee
Then $w$ is a constant map.
\end{prop}
\begin{proof} From the equality $|d^\pi w|^2\, dA = d(w^*\lambda)$ and the hypothesis $E^\pi(w) = 0$ imply
$|d^\pi w|^2 = 0 = d(w^*\lambda)$ in addition to $d(w^*\lambda \circ j) = 0$.
Therefore we derive that $d^\pi w = 0$. This implies
$$
dw = w^*\lambda\, X_\lambda(w)
$$
with $w^*\lambda$ a bounded harmonic one-form. The boundedness of $w^*\lambda$ follows from the hypothesis
$|dw|_{C^0} < \infty$. Since $\C$ is connected, the image of $w$ must be contained
in a single leaf of Reeb foliation. We parameterize the leaf by $\gamma: \R \to Q$,
$\gamma = \gamma(t)$.

Then there is a smooth function $b = b(z)$ such that
$$
w(z) = \gamma(b(z)).
$$
Since $w^*\lambda$ is exact on $\C$, $w^*\lambda = db$ for some function $b$. Since we also have $d(w^*\lambda\circ j) = 0$,
$$
d(db \circ j) = 0
$$
i.e., $b: \C \to \R$ is a harmonic function and hence $b$ is the imaginary part of a holomorphic
function $f$, i.e., $f(z) = a(z) + ib(z)$. Since $b$ has bounded gradient. the gradient of $f$
is also bounded on $\C$. Therefore $f(z) = \alpha z + \beta$ for some constants $\alpha, \, \beta \in \C$.

Once this is achieved, the rest of the argument is exactly the same as Hofer's proof of Lemma 28 \cite{hofer}
via the usage of the $\lambda$-energy bound $E_\infty^\lambda(w) < \infty$ and so omitted.
\end{proof}

Using the above proposition, we prove the following fundamental result.

\begin{thm}\label{thm:C1bound} Let $w: \C \to Q$ be a solution of \eqref{eq:contact-instanton}.
Suppose
$$
E(w) = E^\pi(w) + E^\lambda_\infty(w) < \infty.
$$
Then $|dw|_{C^0} < \infty$.
\end{thm}
\begin{proof}
Suppose to the contrary that $|dw|_{C^0} = \infty$ and let $z_\alpha$ be a blowing-up
sequence. We denote $R_\alpha = |dw(z_\alpha)| \to \infty$. Then by applying Lemma \ref{lem:Hoferslemma},
we can choose another such sequence $z_\alpha'$ and $\epsilon_\alpha \to 0$ such that
\be\label{eq:blowingup-sequence}
|dw(z_\alpha')| \to \infty, \quad \max_{z \in D_{\epsilon_\alpha(z_\alpha')}}|dw(z)| \leq 2 R_\alpha,
\quad \epsilon_\alpha R_\alpha \to 0.
\ee
We consider the re-scaling maps $\widetilde w_\alpha: D^2_{\epsilon_\alpha R_\alpha}(0) \to Q$
defined by
$$
w_\alpha(z) = w \left(z_\alpha' + \frac{z}{R_\alpha}\right).
$$
Then we have
$$
|d w_\alpha|_{C^0; \epsilon_\alpha R_\alpha} \leq 2, \quad |d w_\alpha(0)|=1.
$$
Applying Ascoli-Arzela theorem, there exists a continuous map $w_\infty: \C \to Q$ such that
$ w_\alpha \to w_\infty$ uniformly on compact subsets. Then by the a priori $W^{k,2}$-estimates,
Theorem \ref{thm:Wk2},
the convergence is in compact $C^\infty$ topology and $w_\infty$ is smooth. Furthermore
$w_\infty$ satisfies $\delbar^\pi w_\infty = 0 = d(w_\infty^*\lambda \circ j) = 0$, $E^\lambda(w_\infty) \leq E(w) < \infty$
and
$$
|d w_\infty|_{C^0; \C} \leq 2, \quad |dw_\infty(0)|=1.
$$

On the other hand, by the finite $\pi$-energy hypothesis and density identity $|d^\pi w|^2 \, dA = d(w^*\lambda)$,
we derive
\beastar
0 & = & \lim_{\alpha \to \infty} \int_{D_{\epsilon_\alpha}(z_\alpha')} d(w^*\lambda) =
\lim_{\alpha \to \infty} \int_{D_{\epsilon_\alpha R_\alpha}(z_\alpha')} d( w_\alpha^*\lambda)\\
&= & \lim_{\alpha \to \infty} \int_{D_{\epsilon_\alpha R_\alpha}(z_\alpha')} |d^\pi \widetilde w_\alpha|^2
= \int_\C |d^\pi w_\infty|^2.
\eeastar
 Therefore we derive
$$
E^\pi(w_\infty) = 0.
$$
Then Proposition \ref{prop:C^1} implies $w_\infty$ is a constant map which contradicts to
$|dw_\infty(0)| = 1$. This finishes the proof.
\end{proof}

An immediate corollary of this theorem and Proposition \ref{prop:C^1} is the following

\begin{cor}\label{cor:pi-positive} For any non-constant contact instanton $w: \C \to Q$ with the
energy bound $E(w) < \infty$, we obtain
$$
E^\pi(w) = \int z^*\lambda > 0
$$
for $z = \lim_{R \to \infty} w(R e^{2\pi it})$. In particular $E^\pi(w) \geq T_\lambda > 0$.
\end{cor}

Now we have the following refinement of the
asymptotic convergence result from \cite{hofer} and \cite{oh-wang1}. It is a refinement of Theorem 6.3 of \cite{oh-wang1}
in that the derivative
bound $|dw|_{C^0} < \infty$ imposed therein is replaced by the more natural energy bound $E(w) < \infty$.

\begin{thm}[Compare with Theorem 31 \cite{hofer}, Theorem 6.3 \cite{oh-wang2}]\label{thm:subsequence-limit}
Let $\dot \Sigma$ be a punctured Riemann surface
equipped with a K\"ahler metric that is cylindrical around punctures.
Let $w: \dot \Sigma \to Q$ be a solution of \eqref{eq:contact-instanton}.
Let $p$ be a given puncture. Suppose
\be\label{eq:E-bound}
E(w) < \infty.
\ee
Then for any given sequence $R_i \to \infty$, there exists a subsequence,
again denoted by $R_i$, and a map $w_\infty: \R \times S^1 \to Q$ such
that
\begin{enumerate}
\item for any given $K > 0$, $w_i$ defined by $w_i(\tau,t) = w(\tau + \tau_i,t)$
converges to $w_\infty$ uniformly on $[-K, K] \times S^1$,
\item the image of $w_\infty$ is contained in a single leaf of Reeb foliation. Therefore if we fix a
parametrization of this leaf by $\gamma = \gamma(t)$ for $t \in \R$, then
$$
w_\infty(\tau,t) = \gamma(Q(p)\tau + T(p) t).
$$
\end{enumerate}
Furthermore one of the following alternatives holds: Consider
\bea\label{eq:Ta}
T(p) & = & \int w(0,\cdot)^*\lambda\
+ \frac{1}{2} \int_{[0,\infty) \times S^1} |d^\pi w|^2 = \lim_{i \to \infty} \int w(\tau_i,\cdot)^*\lambda\\
Q(p) & = & - \int_{S^1} (w(0,\cdot))^*\lambda \circ j
\eea
\begin{enumerate}
\item
When $T \neq 0$, there exists a Reeb orbit $\gamma$ of period $T$ such that
$$
w_\infty(\tau,t) =  \gamma(Q(p)\tau + T(p) t)
$$
as $i \to \infty$ where $z_{R_i}(t) = w(\tau_i, t)$ at each puncture $p$, and
its period is given by $T$. In this case, $w(\tau_i,\cdot) \to \gamma(T(\cdot))$
as $i \to \infty$.
\item When $T = 0$, $w_\infty(\tau,t) = \gamma(Q(p)\, \tau)$. In this case,
$w(\tau_i,\cdot)$ converges to a point in the leaf.
\end{enumerate}
\end{thm}

Combining Theorem \ref{thm:subsequence-limit} and Theorem \ref{thm:C1bound}, we immediately derive

\begin{cor} Let $w$ be a non-constant contact instanton on $\C$ with
\be\label{eq:C1-densitybound}
E(w) < \infty.
\ee
Then there exists a sequence $R_j \to \infty$ and a Reeb orbit $\gamma$ such that
$z_{R_j} \to \gamma(T(\cdot))$ with $T \neq 0$ and
$$
T = E^\pi(w), \quad Q = \int_z w^*\lambda \circ j = 0.
$$
\end{cor}
\begin{proof} If $T = 0$, the above theorem shows that there exists a sequence
$\tau_i \to \infty$ such that $w(\tau_i,\cdot)$ converges to a constant in $C^\infty$
topology and so
$$
\int_{\{\tau = \tau_i\}} w^*\lambda \to 0
$$
as $i \to \infty$. By Stokes' formula, we derive
$$
\int_{D_{e^{\tau_i}}(0)} w^*d\lambda = \int_{\tau = \tau_i} w^*\lambda \to 0.
$$
On the other hand, we have
$$
E^\pi(w) = \lim_{i \to \infty} \int_{D_{e^{\tau_i}}(0)} |d^\pi w|^2
 = \lim_{i \to \infty}  \int_{D_{e^{\tau_i}}} w^*d\lambda = 0.
$$
This contradicts to Corollary \ref{cor:pi-positive}, which finishes the proof.
\end{proof}

The following is the analog to Proposition 30 \cite{hofer}.

\begin{cor}\label{cor:C^1oncylinder} Let $w$ be a contact instanton on $\R \times S^1$ with $E(w) < \infty$.
Then $|dw|_{C^0} < \infty$.
\end{cor}
\begin{proof} As in Hofer's proof of Proposition 30 \cite{hofer},
we apply the same kind of bubbling-off argument as that of Theorem \ref{thm:C1bound}
and derive the same conclusion. For readers' convenience, we provide the details of
the proof in Appendix \ref{append:C^1oncylinder}.
\end{proof}

\section{Bubbling-off analysis and the period-gap theorem }
\label{sec:e-regularity}

We recall from \cite{oh-wang2} that the local a priori $W^{k,2}$-regularity estimates are established with respect to
the bounds of $\|dw\|_{L^4}$ and $\|dw\|_{L^2}$. Therefore in addition to the local a priori $W^{k,2}$-regularity estimates,
one should establish another crucial ingredient, the $\epsilon$-regularity result, for the
study of moduli problem as usual in any of conformally invariant geometric non-linear PDE's.
This will in turn establish the $W^{1,p}$-bound
with $p > 2$ (say $p = 4$) appears in many problems in geometry and physics
under the suitable smallness hypothesis on the relevant energy. (See \cite{sacks-uhlen}.)

In the current setting of contact instanton map, it is not obvious what would be
the precise form of relevant $\epsilon$-regularity statement is.
We formulate this $\epsilon$-regularity theorem in the setting of contact instantons.
It turns out that the relevant energy is the $\pi$-harmonic energy.

\begin{defn} Let $\lambda$ be a contact form of contact manifold $(Q,\xi)$.
Denote by $\frak R eeb(Q,\lambda)$ the set of closed Reeb orbits.
We define $\operatorname{Spec}(Q,\lambda)$ to be the set
$$
\operatorname{Spec}(Q,\lambda) = \left\{\int_\gamma \lambda \mid \lambda \in \frak Reeb(Q,\lambda)\right\}
$$
and call the \emph{action spectrum} of $(Q,\lambda)$. We denote
$$
T_\lambda: = \inf\left\{\int_\gamma \lambda \mid \lambda \in \frak Reeb(Q,\lambda)\right\}.
$$
\end{defn}

We set $T_\lambda = \infty$ if there is no closed Reeb orbit.

The following is a standard lemma in contact geometry

\begin{lem} Let $(Q,\xi)$ be a closed contact manifold. Then
$\operatorname{Spec}(Q,\lambda)$ is either empty or a countable nowhere dense subset of $\R_+$
and $T_\lambda > 0$. Moreover the subset
$$
\operatorname{Spec}^{K}(Q,\lambda) = \operatorname{Spec}(Q,\lambda) \cap (0,K]
$$
is finite for each $K> 0$.
\end{lem}

\begin{rem}
A priori we cannot rule out the possibility $\operatorname{Spec}(Q,\lambda)  = \emptyset$.
Nonemptyness of this set is precisely the content of Weinstein's conjecture:
Any contact form $\lambda$ of a contact manifold $(Q,\xi)$ carries a closed Reeb orbit.
The conjecture has been proved by Taubes \cite{taubes} in 3 dimensional case
after other scattered results obtained earlier.
\end{rem}

The constant $T_\lambda$ will enter in a crucial way in the following
$\epsilon$-regularity statement.  The proof of this theorem will closely follow the  argument used in
\cite[section 8.4]{oh:book} and \cite{oh:imrn} by adapting it to
the proof of the current $\epsilon$-regularity theorem with the replacement of
the standard harmonic energy by the $\pi$-harmonic energy. However there is one
marked difference between the current $\epsilon$-regularity statement and that of
pseudoholomorphic curves because of the second order part $d(w^*\lambda\circ j) = 0$
of contact instanton map: The local $W^{k,2}$ a priori estimate given in
Theorem \ref{thm:Wk2} plays a crucial role in establishing that the limit
map of a subsequence obtained via application of Ascoli-Arzela theorem still
satisfies the equation $\delbar^\pi w = 0, \, d(w^*\lambda \circ j) = 0$.

\begin{thm}\label{thm:e-regularity}
Denote by $D^2(1)$ the closed unit disc.
Let $w:D^2(1)\to Q$ satisfy
$$
\delbar^\pi w = 0, \, d(w^*\lambda\circ j) = 0,\, E^\lambda(w) < K_0.
$$
Then for any given $0 <\epsilon < T_\lambda$ and $w$ satisfying
$E^\pi(w) < T_\lambda - \epsilon$, and for a smaller disc $D' \subset \overline D' \subset D$,
there exists some $K_1 = K_1(D', \epsilon,K_0) > 0$
\be\label{eq:dwC0}
\|dw\|_{C^0;D'} \leq K_1
\ee
where $K_1$ depends only on $(Q,\lambda,J)$, $\epsilon$, $D' \subset D$.
\end{thm}
\begin{proof}
Suppose to the contrary that
there exists a disc $D' \subset D$ with $\overline{D'} \subset \overset{\circ}D$ and a sequence $\{ w_\alpha \}$
such that
$$
\delbar^\pi w_\alpha = 0, \quad d(w_\alpha \circ j) = 0
$$
and satisfy
\be\label{eq:2to0ptoinfty}
E^\pi_{\lambda,J;D}(w_\alpha) < T_\lambda - \epsilon,, E^\lambda(w_\alpha) < K_0,
\quad \norm{dw_\alpha}{C^0,D'} \to \infty
\ee
as $\alpha \to \infty$. Let $x_\alpha \in D'$ such that $|dw_\alpha(x_\alpha)| \to \infty$.
By choosing a subsequence, we may assume that $x_\alpha \to x_\infty\in \overline D' \subset \overset{\circ} D$.
We take a coordinate chart centered at $x_\infty$ on $D_{x_\infty}(\delta) \subset \overset{\circ} D$
and identify $D_{x_\infty}(\delta)$ with the disc $D^2(\delta) \subset \C$ and $x_\infty$ with $0 \in \C$.
This can be done by choosing $\delta > 0$ sufficiently small since we assume $\overline D' \subset \overset{\circ} D$.
Then $x_\alpha \to 0$. We choose $\delta_\alpha \to 0$ so that $\delta_\alpha |dw_\alpha(x_\alpha)| \to \infty$.

We adjust the sequence $x_\alpha$ to $y_\alpha$ by applying
Hofer's lemma, Lemma \ref{lem:Hoferslemma}, so that
$y_\alpha \to 0$ and
\be\label{eq:adjustedy}
\max_{x \in B_{y_\alpha}(\epsilon_\alpha)}|dw_\alpha| \leq 2|dw_\alpha(y_\alpha)|, \quad
\delta_\alpha |dw_\alpha(y_\alpha)| \to \infty.
\ee
We denote $R_\alpha =  |dw_\alpha(y_\alpha)|$ and consider the re-scaled map
$$
v_\alpha(z) = w_\alpha\left(y_\alpha + \frac{z}{R_\alpha}\right).
$$
Then the domain of $w_\alpha$ at least includes $z \in \C$ such that
$$
y_\alpha + \frac{z}{R_\alpha} \in D^2(\delta),
$$
i.e., those $z$'s satisfying
$$
\left|y_\alpha + \frac{z}{R_\alpha} \right| \leq \delta.
$$
In particular, if $|z| \leq R_\alpha (\delta - |y_\alpha|)$, $v_\alpha(z)$ is
defined. Since $y_\alpha \to 0$ and $\delta_\alpha \to 0$ as $\alpha \to \infty$,
$R_\alpha (\delta - |y_\alpha|)> R_\alpha \epsilon_\alpha$
eventually, $v_\alpha$ is defined on $D^2(\epsilon_\alpha R_\alpha)$ for all sufficiently
large $\alpha$'s.
Since $\delta_\alpha R_\alpha \to \infty$ by \eqref{eq:adjustedy}, for any given $R>0$,
$D^2(\delta_\alpha R_\alpha)$ of $v_\alpha (z)$ eventually contains $B_{R+1}(0)$.

Furthermore, we may assume,
$$
B_{R+1}(0) \subset \left\{ z \in \mathbb{C} \mid \eta_\alpha z + y_\alpha \in \overline{D}'\right\}
$$
Therefore, the maps
$$
v_\alpha : B_{R+1} (0) \subset \mathbb{C} \to M
$$
satisfy the following properties:
\begin{enumerate}
 \item[(i)]  $E^\pi(v_\alpha) < T_\lambda -\epsilon$, \, $\delbar^\pi v_\alpha = 0$, \,
 $E^\lambda(v_\alpha) \leq K_0$,
 (from the scale invariance)
 \item[(ii)] $|dv_\alpha(0)|=1 $ by definition of $v_\alpha$ and $R_\alpha$,
 \item[(iii)]$\norm{dv_\alpha}{C^0,B_1(x)} \leq 2$ for all $x \in B_R(0) \subset D^2(\epsilon_\alpha R_\alpha)$,
 \item[(iv)] $\delbar^\pi v_\alpha =0$ and $d(v_\alpha^*\lambda \circ j) = 0$.
\end{enumerate}

For each fixed $R$, we take the limit of $v_\alpha|_{B_R}$, which we denote by $w_R$.
Applying (iii) and then the local $W^{k,2}$ estimates, Theorem \ref{thm:Wk2},  we obtain
$$
\norm{dv_\alpha}{k,2;B_{\frac9{10}}(x)} \leq C
$$
for some $C=C(R)$. By the Sobolev embedding theorem, we have a
subsequence that converges in $C^2$ in each $B_{\frac{8}{10}}(x), x \in \overline D'$. Then
we derive that the convergence is in $C^2$-topology on $B_{\frac{8}{10}}(x)$ for all $x \in \overline D'$ and
in turn on $B_R(0)$.

Therefore the limit $w_R: B_R(0) \to M$ of $v_\alpha|_{B_R(0)}$ satisfies
\begin{itemize}
\item[(1)] $E^\pi(w_R) \leq T_\lambda -\epsilon$, $\delbar^\pi w_R = 0$,
$d(w_R^*\lambda \circ j) = 0$ and $E^\lambda(v_\alpha) \leq K_0$,
\item[(2)] $E^\pi(w_R) \leq \limsup_\alpha  E^\pi_{(\lambda,J;B_R(0))}(v_\alpha) \leq T_\lambda -\epsilon$,
\item[(3)]  Since $v_\alpha \to w_R$ converges in $C^2$, we have
$$
\norm{dw_R}{p,B_1(0)}^2 = \lim_{\alpha \to \infty} \norm{dv_\alpha}{p,B_1(0)}^2 \geq \frac{1}{2}.
$$
\end{itemize}
By letting $R \to \infty$ and taking a diagonal subsequence argument, we have
derived nonconstant contact instanton map $w_\infty: \C \to Q$. Therefore by definition of $T_\lambda$,
we must have $E^\pi(w_\infty) \geq T_\lambda$.

On the other hand, the bound
$E^\pi(w_R) \leq T_\lambda -\epsilon$ for all $R$ and again by Fatou's lemma implies
$$
E^\pi(w_\infty) \leq T_\lambda -\epsilon
$$
which gives rise to a contradiction.
This finishes the proof of \eqref{eq:dwC0}.
\end{proof}
%
%

%

\section{Asymptotic behaviors of finite energy contact instantons}
\label{sec:asymptotic}

%
In this section, we study the asymptotic behavior of contact instanton
$w:\dot \Sigma \to Q$ with finite energy $E(w) < \infty$ near the punctures.
We start with classifying the solutions of \eqref{eq:contact-instanton}
of zero energy on the cylinder $\R \times S^1$.

We start with the following lemma

\begin{lem}\label{lem:|dw|<infty} Suppose $E(w) = E^\pi(w) + E^\lambda(w) < \infty$. Then
$$
|dw|_{C^0} < \infty.
$$
\end{lem}
\begin{proof} By the finiteness $E^\pi(w)< \infty$, we can choose
sufficiently small $\delta > 0$ such that
$$
E^\pi(w|_{\Sigma\setminus \Sigma(\delta)})< \frac{1}{2}T_\lambda.
$$
Denote
$$
\Sigma(\delta) = \dot \Sigma \setminus \cup_{\ell =1}^k D_{r_\ell}(\delta).
$$
Then we apply the $\epsilon$-regularity theorem,
Theorem \ref{thm:e-regularity}, to $w$ on $\cup_{\ell =1}^k D_{r_\ell}(\delta) = \dot \Sigma \setminus \Sigma(\delta)$
to derive
$$
|dw|_{\cup_{\ell =1}^k D_{r_\ell}} < \infty.
$$
Obviously $|dw|_{\Sigma(\delta)}|_{C^0} < \infty$ and hence the proof.
\end{proof}

\subsection{Massless contact instantons}

The following is a key lemma in which the closed condition $w^*\lambda \circ j$
plays a crucial role.

\begin{lem}\label{lem:massless}
Let $\dot \Sigma$ be any punctured Riemann surface.
Suppose $w: \dot \Sigma \to Q$ is a massless contact instanton
on $\dot \Sigma$. Then $w^*\lambda $ is a harmonic 1-form and
the image of $w$ lies in a single leaf of the Reeb foliation.
\end{lem}
\begin{proof} From the equation, we have $\delbar^\pi w = 0$. We also
have $\del^\pi w = 0$ from the massless condition and so $d^\pi w = \pi dw = 0$.
This implies the values of $dw$ are parallel to $X_\lambda$ at all points of $\dot \Sigma$.
By the connectedness of $\dot \Sigma$, this implies that the image of $w$ must be contained in
a leaf.

Next we obtain $d(w^*\lambda) = 0$ from $E_{(\lambda,J)}^\pi(w) = 0$ and the identity
$|d^\pi w|^2 = |\del^\pi w|^2\, dA = d(w^*\lambda)$ since $\delbar^\pi w = 0$.
We also have
$$
\delta (w^*\lambda)\, dA = - d(w^*\lambda \circ j) = 0
$$
where the first follows since the metric $h$ on $\dot\Sigma$
is K\"ahler with respect to $j$ and the second equality follows from the equation.
This finishes the proof.
\end{proof}

The following result connects the basic hypotheses for the a priori $W^{k,2}$-estimates
to the study of structure of singularities of contact instanton.

\begin{prop}\label{prop:pole-structure}
Let $w$ be a contact instanton on $\dot \Sigma$ with punctures $p \in \{p_1, \cdots, p_k\}$.
Let $p \in \{p_1, \cdots, p_k\}$ and let $z$ be an analytic coordinate at $p$. Suppose
$$
E(w)= E^\pi(w) + E^\lambda(w) < \infty.
$$
Then for any given sequence $\delta_j \to 0$ there exists a subsequence, still denoted by $\delta_j$, and a conformal diffeomorphism
$\varphi_j: [-\frac{1}{\delta_j}, \infty) \times S^1 \to D_{\delta_j}(p)\setminus \{p\}$ such that
the one form $\varphi_j^*\chi$ converges to a bounded holomorphic one-form
$\chi_\infty$ on $(-\infty, \infty) \times S^1$.
\end{prop}
\begin{proof} By Lemma \ref{lem:|dw|<infty}, $|dw|_{C^0} < \infty$. Let $C =  |dw|_{C^0}$.
Then $|w^*\lambda|_{C^0} \leq C$.

By the finiteness $E^\pi(w) < \infty$, Fatou's lemma implies
$$
\lim_{r\to 0} \int_{D_r(p)\setminus \{0\}} |d^\pi w|^2 = 0.
$$
We fix a sequence $r_j \to 0$ and fix a conformal diffeomorphism
$$
\varphi_j:  \left[-\frac{1}{\delta_j}, \infty\right) \times S^1
\to D_{r_j}(p)\setminus \{0\}, \quad \varphi_r(\tau,t) = \delta_0 e^{-\frac1{\delta_j}} e^{-2\pi(\tau +it)} = z
$$
for each $j > 0$.  In particular,
the map $(\varphi_j^*w,\varphi_j^*\chi)$ are contact-instantons on $[0,\infty) \times S^1$ which satisfy
$$
E^\pi(\varphi_j^*w) \to 0.
$$
By $W^{k,2}$ a priori estimates, Theorem \ref{thm:Wk2}, and the $\epsilon$-regularity theorem, Theorem \ref{thm:e-regularity},
we obtain the gradient bound $|d(\varphi_j^*w)|_{[-1/\delta_j,\infty) \times S^1} \leq C$
and in particular $|(\varphi_j^*w)^*\lambda|_{C^0} \leq C$ for all $j$.

Applying the diagonal subsequence argument,
we can select a sequence $\delta_j \to 0$ such that
$\varphi_{\delta_j}^*w$ converges to $w_\infty: (-\infty, \infty) \times S^1 \to Q$
and $\varphi_{\delta_j}^*\chi \to \chi_\infty$
in compact $C^\infty$ topology so that the pair $(w_\infty,\chi_\infty)$ is a
contact instanton satisfying
\be\label{eq:massless}
E^\pi(w_\infty)= 0, \quad |\chi_\infty|_{C^0} \leq \frac{3C}{2}.
\ee
Since $|d^\pi w_\infty|^2 \, dA = d(w_\infty^*\lambda)$, this implies
$$
d(w_\infty^*\lambda) = 0.
$$
Together with $d(w_\infty^*\lambda\circ j) = 0$, this implies that
$\chi_\infty$ is a non-zero holomorphic one-form that is bounded on $\R \times S^1$.
This finishes the proof.
\end{proof}

We would like to emphasize that at the moment, the limiting holomorphic one-form $\chi_\infty$
may depend on the choice of subsequence.

The following theorem slightly strengthens the convergence results from \cite{hofer},
\cite{oh-wang2}.

\begin{thm}\label{thm:nondegeneratelimit}
Let $\Sigma$ be a closed Riemann surface of genus 0 with
a finite number of marked points $\{p_1, \cdots, p_k\}$ for $k \geq 3$, and let $\dot \Sigma = \Sigma \setminus
\{p_1,\cdots, p_k\}$ be the associated punctured Riemann surface equipped with a metric as before.
Suppose that $w$ is a contact instanton map $w:(\dot \Sigma,j)
\to (Q,J)$ with finite total energy $E(w) = E^\pi(w) + E^\lambda(w)$
and fix a puncture $p \in \{p_1, \cdots, p_k\}$.

Then for any given sequence $I=\{\tau_k\}$ with $\tau_k \to \infty$,
there exists a subsequence $I' \subset I$ and a closed parameterized Reeb orbit $\gamma = \gamma_{I'}$ of period $T$
and some $(\tau_0,t_0) \in \R \times S^1$ such that
such that
$$
\lim_{i \to \infty} w(\tau + \tau_{k_i},t) = \gamma(Q(p)\, \tau + T(p)\, t)
$$
in compact $C^\infty$ topology.

If $\lambda$ is nondegenerate and $T \neq 0$, then the convergence $w(\tau,\cdot) \to \gamma(T\cdot)$
is uniform.
\end{thm}
\begin{proof} The finiteness of $E(w)$ and the $\epsilon$-regularity implies the $C^1$ bound $|dw|_{C^0} < \infty$
on $[R,\infty) \times S^1$ for a sufficiently large $R> 0$.
Once this bound is established, the same proof
as that of Theorem 6.3 of \cite{oh-wang2} proves that
there exists a closed Reeb orbit $(T,\gamma)$ and a subsequence
$k_i \to \infty$ such that
$$
w(\tau_{k_i} +\tau, \cdot) \to \gamma(Q(p)(\tau_{k_i} + \tau), T(p)t)
$$
uniformly on $[-K,K] \times S^1$ in $C^\infty$ topology for any given $K \geq 0$.
Once we have established this subsequence convergence result,
the same proof as that of Theorem 6.5 \cite{oh-wang2} applies to conclude the theorem.
We refer to \cite{oh-wang2} for the complete detail of the proof and the proof of
uniform convergence for the nondegenerate case.
\end{proof}

We would like to call the readers' attention to the case where $T(p) = 0$. In this case
the asymptotic limit $w_\infty$ is $t$-independent, i.e., $w_\infty(\tau,t) \equiv \gamma(Q(p) \tau)$.
In particular, the image of the instanton is 1 dimensional.

\subsection{Classification of punctures}
\label{subsec:puncture-types}

Assume that $\lambda$ is nondegenerate.
We would like to further analyze the asymptotic behavior of the instanton $w$.

Associated to the splitting
$$
TQ = \span\{X_\lambda\} \oplus \xi,
$$
$Q$ carries the canonical (trivial) complex line bundle $\CL \to Q$ with
connection form $\sqrt{-1} \lambda$. When we are given a map
$w: \dot \Sigma \to Q$, it induces the pull-back bundle
$w^*\CL$ with the pull-back connection $\sqrt{-1} w^*\lambda$.
The associated  (abelian) Yang-Mills equation is nothing but
$$
\delta w^*\lambda = 0
$$
with respect to the K\"ahler metric associated to the complex structure $j$ on
the surface $\Sigma$ is precisely equivalent to $d(w^*\lambda \circ j) = 0$.

Now we introduce the complex valued one-form
\be\label{eq:specialchi}
\chi =  w^*\lambda \circ j + \sqrt{-1} w^*\lambda.
\ee
It appears to be worthwhile to give a name to the complex valued $(1,0)$-form in
the general context.

\begin{defn}\label{defn:charge} Let $(\Sigma,j)$ be a closed
Riemann surface with finite number of marked points $\{p_1, \cdots, r_k\}$. Denote by $\dot \Sigma$ the
associated punctured Riemann surface with cylindrical metric near the punctures, and let $\overline \Sigma$
the real blow-up of $\Sigma$ along the punctures.
Let $w$ be a contact instanton map. Let $p \in \{p_1, \cdots, r_k\}$. We call the integrals
\bea
Q(p) &: =& - \int_{\del_{\infty;r}\Sigma} w^*\lambda \circ j\\
T(p) &: =& \int_{\del_{\infty;r}\Sigma} w^*\lambda
\eea
the \emph{contact instanton charge} and \emph{contact instanton action} at $p$ respectively.
Here $\del_{\infty;r}\Sigma$ is the
boundary component corresponding to $p$ of the real blow-up $\overline \Sigma$ of $\dot \Sigma$.
Then we call the form $\chi = w^*\lambda \circ j + \sqrt{-1} w^*\lambda$ the \emph{
contact Hick's field} of $w$ and
$$
Q(p) + \sqrt{-1} T(p)
$$
the \emph{charge} of the Hick's field of the instanton $w$ at the puncture $p$.
\end{defn}

Note that by the closedness $d(w^*\lambda \circ j) = 0$, the charge $Q(p)$ is the same as
the initial integral
$$
\int_{\{\tau = 0\}} w^*\lambda \circ j
$$
which does not depend on the choice of subsequence but is determined by
the initial condition at $\tau = 0$ and homology class of the loop
$w|_{\tau = 0} \in H_1(\dot \Sigma) = H_1(\Sigma \setminus \{p_1, \cdots, p_k\}$.

\begin{prop}\label{prop:balancing}
For any finite energy contact instanton $w$, we have
\be\label{eq:totalcharge=0}
\sum_{l =1}^N Q(p_\ell) = 0.
\ee
We call this equation the \emph{balancing condition} of the contact Hick's charge.
\end{prop}
\begin{proof} This is an immediate consequence of Stokes' formula applied to the closed 1-form $w^*\lambda \circ j$ on
the real blow-up $\overline \Sigma$ of $\dot \Sigma$.
\end{proof}

Now we consider the  asymptotic Hick's field $\chi_\infty$ associated to the asymptotic
instanton $w_\infty$ obtained in the proof of Proposition \ref{prop:pole-structure}, and
call $\chi_\infty$ the asymptotic Hick's field of $w$ at the puncture $p$.
Because $w_\infty$ is massless and has bounded derivatives on $\R \times S^1$,
$\chi_\infty$ becomes a bounded holomorphic one-form. Therefore we derive
\be\label{eq:chiinfty}
\chi_\infty = c\, (d\tau + i \, dt)
\ee
for some complex number $c \in \C$. We denote $c = b + i a$ for $a, \, b \in \R$.
Equivalently, we obtain
$$
w^*\lambda = a\, d\tau + b\, dt.
$$
Here $a, \, b$ are nothing but the period integrals
$$
a = - \int_{S^1} (w(\tau,\cdot))^*\lambda \circ j, \, b =  \int_{S^1} (w(\tau,\cdot))^*\lambda
$$
which do not depend on $\tau$ for the massless instantons,
thanks to the closedness of $w^*\lambda, \, w^*\lambda \circ j$.
We denote them by $a = Q(p), \, b = T(p)$ and call them as the \emph{Hick's charge} at $p$.

%
%
We now examine the various cases arising depending on the constant $c$.
Let $\chi_\infty = c\, (d\tau + i \, dt)$ as above.

\begin{thm} \label{thm:c=0}
Suppose $c = 0$. Then $w$ is smooth across $p$ and so the
puncture $p$ is removable.
\end{thm}
\begin{proof} When $c = 0$, we obtain $dw_\infty = d^\pi w_\infty
+ \lambda^*w_\infty\, X_\lambda = 0$ and so
$w_\infty$ must be a constant map $q \in Q$. By the convergence $w_j \to w_\infty$ in compact $C^\infty$ topology,
it follows that $w_j(0,\cdot) \to q$ or equivalently
$$
d(w|_{r = \delta_j}, q) \to 0
$$
and $w_j^*\lambda \to 0$ converges uniformly. Using the compactness of $Q$ and applying Ascoli-Arzela theorem,
we can choose a sequence $z_i \to p$ in $D_\delta(p) \setminus \{p\}$ such that
$w(z_i) \to p$ and $w^*\lambda|_{r = \delta_j} \to 0$ uniformly. Then this
continuity of $w^*\lambda$ at $p$ in turn implies $dw$ is continuous at $p$
by the expression
$$
dw = d^\pi w  + w^*\lambda\, X_\lambda(w)
$$
In particular $|dw|_{D_\delta(r)}$ is bounded and so lies in
$L^2 \cap L^4$ on $D_\delta(r)$.  Then the local $W^{k,2}$ a priori estimate implies
that $w$ is indeed smooth across $p$. This finishes the proof.
\end{proof}

\medskip

If $c \neq 0$, we obtain
$$
\int_{S^1} \chi_\infty|_{\tau} \equiv c
$$
for all $\tau$. In particular, we derive
$$
\lim_{j \to \infty} \int_{S^1} (\chi|_{r = \delta_j})^*\lambda = c
$$
and so
\beastar
\lim_{k \to \infty}\int_{S^1} (w|_{r = \delta_k})^*\lambda \circ j & = & \Re c\\
\lim_{k \to \infty}\int_{S^1} (w|_{r = \delta_k})^*\lambda & = & \Im c.
\eeastar
In fact by the closedness of $w^*\lambda \circ j$ and convergence of $w|_{r = \delta_j} \to p$,
the integral  $(w|_{r = \delta_k})^*\lambda \circ j$ does not depend on $k$'s eventually.

We divide our consideration of the remaining cases into two different cases,
one with $b = \Im c = 0$ and the other with $b = \Im c \neq 0$.

\begin{prop}\label{prop:bnot0} Suppose $b \neq 0$. Then there exists a closed Reeb orbit $\gamma$ of
period $T = \frac{b}{2\pi}$ such that there exists a sequence $\tau_k \to \infty$
for which $w(\tau_k, \cdot) \to \gamma(T(\cdot))$ in $C^\infty$ topology.
\end{prop}
\begin{proof}
When $b \neq 0$, we obtain
$$
dw_\infty = (a\, d\tau + b \, dt)\, X_\lambda.
$$
Again by the connectedness of $[0,\infty) \times S^1$, it follows that
the image of $w_\infty$ must be contained in a single leaf of the Reeb foliation
and so
$$
w_\infty(\tau,t) = \gamma(a\tau + b\, t)
$$
for a parameterized Reeb orbit $\gamma$ such that $\dot \gamma = X_\lambda(\gamma)$.
Such a parameterization is unique modulo the time-shift.
Since the map $w$ is one-periodic for any $\tau$, we derive
$$
\gamma(b) = \gamma(0).
$$
This implies first that $\gamma$ is a periodic Reeb orbit of period $g$.
\end{proof}

If we denote by $T > 0$ its minimal
period, then we obtain
$$
2\pi\, b = m\, T
$$
for some integer $m$. Since we assume $b \neq 0$, it follows that $m\, T  \neq 0$.

\begin{prop}\label{prop:b=0anot0} Suppose $b = 0, \, a\neq 0$.
Then $w_\infty$ does not depend on the $t$-variable and the map $\tau \to w_\infty(\tau)$ becomes a Reeb trajectory which
is not necessarily closed.
\end{prop}
\begin{proof} In this case, $w_\infty^*\lambda = a\, d\tau$. Therefore $w_\infty$ does
not depend on $t$ and satisfies
$$
\frac{\del w_\infty}{\del \tau} = a\, X_\lambda(w(\tau,t))
$$
and so $w(\tau,t) \equiv z (a \tau)$ for a path satisfying $\dot z = X_\lambda(z)$.
This finishes the proof.
\end{proof}

\begin{rem}
\begin{enumerate}
\item  We would like to remark that all the above three scenarios
can actually occur and have to be examined in the asymptotic study of
contact instantons. For the exact case, we have $a = 0$.
\item Each massless contact instanton on $\R \times S^1$
induces a linear foliation thereon. When the charge is zero, the foliation becomes
the standard foliation but when the instanton carries a non-trivial charge
the `horizontal' foliation is skewed. This could be interpreted as the change of conformal
structure (or `gravity' by physical terms) of the cylinder that is
powered by non-trivial charge carried by the instanton.
This phenomenon seems to be worthwhile to further study which is a subject of future study.
\item Presence of the above non-trivial `spiraling' massless instantons on the cylinder which
does not exist in the exact case, makes the asymptotic study of contact instantons for the non-exact case
more complicated but also makes more interesting.
\end{enumerate}
\end{rem}

Now we are ready to define the notion of positive and negative
punctures of contact instanton map $w$. Assume $\lambda$ is nondegenerate.

Let $p$ be one of the punctures of $\dot \Sigma$.
In the disc $D_\delta(p) \subset \C$ with the standard orientation, we consider the function
$$
\int_{\del D_\delta(p)} w^*\lambda
$$
as a function of $\delta > 0$. This function is either decreasing or increasing
by the Stokes' formula, the positivity $w^*d\lambda \geq 0$ and the finiteness of
$\pi$-energy
$$
\frac{1}{2} \int_{\dot \Sigma} |d^\pi w|^2 = \int_{\dot \Sigma} w^*d\lambda < \infty.
$$
\begin{defn}[Classification of punctures]
Let $\dot \Sigma$ be a puncture Riemann
surface with punctures $\{p_1, \cdots, p_k\}$ and let
$w: \dot \Sigma \to Q$ be a contact instanton map.
\begin{enumerate}
\item We call a puncture $p$ \emph{removable} if $T(p) = Q(p) = 0$, and \emph{non-removable} otherwise.
Among the non-removable punctures $p$, we call it
\emph{non-adiabatic} if $T(p) \neq 0$, \emph{adiabatic} if $T(p) = 0$ but $Q(p) \neq 0$.
\item
We say a non-removable puncture \emph{positive} (resp. \emph{negative}) puncture if the function
$$
\int_{\del D_\delta(p)} w^*\lambda
$$
is increasing (resp. decreasing) as $\delta \to 0$.
\end{enumerate}
\end{defn}

The appearance of adiabatic punctures is a new phenomenon when the form $w^*\lambda \circ j$ is
not exact. In the latter case considered via
the case of symplectization picture \cite{hofer},
the associated puncture is removable and can be dropped in this classification by
removing the puncture. However in the non-exact case, such a puncture is
not necessarily removable and so has to be considered separately.

\section{Properness of contact instanton potential function and $\lambda$-energy}
\label{sec:pi-lambda-energy}

In this section, we examine the relationship between the $\pi$-energy,
the $\lambda$-energy and the contact instanton potential function $f$.

We first note that the function $f: \dot \Sigma \to \R$ is proper
if and only if
\be\label{eq:f-proper}
f(v_j) = \pm \infty
\ee
for all exterior vertex $v_j \in V(T)$.
One immediate corollary of Lemma \ref{lem:|dw|<infty}
 is the following $C^1$-bound of the contact potential function $f$.

\begin{cor} Suppose that $E(w) < \infty$ and let $f$ be the function defined
in section \ref{sec:offshellenergy}. Then $|df|_{C^0} < \infty$.
\end{cor}
\begin{proof}
From Lemma \ref{lem:|dw|<infty} and the defining equation of $f$
$$
w^*\lambda \circ j + \sum_{e \in E(T)} Q(w;e)\, dt_e = df,
$$
we obtain $|df|_{C^0} < |dw|_{C^0} + \max_{e \in E(T)}|Q(w;e)| < \infty$.
\end{proof}

The following proposition is the analog to Lemma 5.15 \cite{behwz} whose proof is also
similar.

\begin{prop}\label{prop:proper-energy} Suppose that $E^\pi(w) < \infty$ and the function $f:\dot \Sigma \to \R$ is proper.
Then $E(w) < \infty$.
\end{prop}
\begin{proof}
Since $f$ is assumed to be proper, $f(r_\ell) = \pm \infty$
for each puncture $r_\ell$ of $\dot \Sigma$ depending on whether the puncture is positive or
negative.

The rest of the argument is very similar to that of the proof of Lemma 5.15 \cite{behwz}
with replacement of $a$ and the equation $dw^*\lambda \circ j = da$ therein by $f$
and the equation
$$
dw^*\lambda \circ j + \sum_{e \in E(T)} Q(w;e)\, dt_e = df
$$
respectively in our current context. (We would also like point out that \cite{behwz}
used the letter `$f$' for the map $w$ which should not confuse the readers with
our notation $f$ for the function which corresponds to $a$ in their notation.)

Since our setting does not use the setting of symplectization, we provide the full details of
the proof in Appendix.
\end{proof}

By the same argument as the derivation of Lemma 5.16 \cite{behwz}, we obtain
\begin{lem} Suppose $E^\pi(w) < \infty$ and $f$ is proper. Denote by
$\gamma^+_1, \cdots, \gamma^+_k$ (resp. $\gamma^-_1,\cdots, \gamma^-_\ell$)
the periodic orbits of $X_\lambda$ asymptotic to the positive (resp. negative punctures) of
$\dot \Sigma$. Then
\beastar
E^\pi(w) & = & \sum_{j=1}^k \int \overline\gamma_j^*\lambda - \sum_{i=1}^\ell \int \underline\gamma_i^*\lambda\\
E^\lambda(w) & = & \sum_{j=1}^k \int \overline\gamma_j^*\lambda \\
E(w) & = & 2 \sum_{j=1}^k \int \overline\gamma_j^*\lambda - \sum_{i=1}^\ell \int \underline\gamma_i^*\lambda.
\eeastar
\end{lem}

\section{Calculation of the linearization map with contact triad connection}
\label{sec:linearization-map}

Let $\Sigma$ be a closed Riemann surface and $\dot \Sigma$ be its
associated punctured Riemann surface.  We allow the set of whose punctures
to be empty, i.e., $\dot \Sigma = \Sigma$.
We would like to regard the assignment
$$
w \mapsto \left(\delbar^\pi w, d(w^*\lambda \circ j)\right)
$$
for a map $w: \dot \Sigma \to Q$ as a section of the (infinite dimensional) vector bundle
over the space of maps of $w$. In this section, we lay out the precise relevant off-shell framework
of functional analysis.

Let $(\dot \Sigma, j)$ be a punctured Riemann surface, the set of whose punctures
may be empty, i.e., $\dot \Sigma = \Sigma$ is either a closed or a punctured
Riemann surface. We will fix $j$ and its associated K\"ahler metric $h$.

We consider the map
$$
\Upsilon(w) = \left(\delbar^\pi w, d(w^*\lambda \circ j) \right)
$$
which defines a section of the vector bundle
$$
\CH \to \CF = C^\infty(\Sigma,Q)
$$
whose fiber at $w \in C^\infty(\Sigma,Q)$ is given by
$$
\CH_w: = \Omega^{(0,1)}(w^*\xi) \oplus \Omega^2(\Sigma).
$$
We decompose $\Upsilon = (\Upsilon_1,\Upsilon_2)$ where
\be\label{eq:upsilon1}
\Upsilon_1: \Omega^0(w^*TQ) \to \Omega^{(0,1)}(w^*\xi); \quad \Upsilon_1(w) = \delbar^\pi(w)
\ee
and
\be\label{eq:upsilon2}
\Upsilon_2: \Omega^0(w^*TQ) \to \Omega^2(\dot \Sigma); \quad \Upsilon_2(w) = d(w^*\lambda \circ j).
\ee

We first compute the linearization map which defines a linear map
$$
D\Upsilon(w): \Omega^0(w^*TQ) \to \Omega^{(0,1)}(w^*\xi) \oplus \Omega^2(\Sigma)
$$
where we have
$$
T_w \CF = \Omega^0(w^*TQ).
$$
We note
\beastar
\operatorname{rank}\Lambda^0(w^*TQ) & = & 2n+1 \\
\operatorname{rank} \Lambda^{(0,1)}(w^*\xi) \oplus \Lambda^2(\Sigma) & = & 2n+1.
\eeastar
For the optimal expression of the linearization map and its relevant
calculations, we use the contact triad connection $\nabla$ of $(Q,\lambda,J)$ and the contact
Hermitian connection $\nabla^\pi$ for $(\xi,J)$ introduced in \cite{oh-wang2}.

\begin{thm}\label{thm:linearization} In terms of the decomposition $d\pi = d^\pi w + w^*\lambda\, X_\lambda$
and $Y = Y^\pi + \lambda(Y) X_\lambda$, we have
\bea
D\Upsilon_1(w)(Y) & = & \delbar^{\nabla^\pi}Y^\pi + B^{(0,1)}(Y^\pi) +  T^{\pi,(0,1)}_{dw}(Y^\pi)\\
&{}& \quad + \frac{1}{2}\lambda(Y) (\CL_{X_\lambda}J)J(\del^\pi w)
\label{eq:Dwdelbarpi}\\
D\Upsilon_2(w)(Y) & = &  - \Delta (\lambda(Y))\, dA + d((Y^\pi \rfloor d\lambda) \circ j)
\label{eq:Dwddot}
\eea
where $B^{(0,1)}$ and $T_{dw}^{\pi,(0,1)}$ are the $(0,1)$-components of $B$ and
$T_{dw}^{\pi,(0,1)}$, where $B, \, T_{dw}^\pi: \Omega^0(w^*TQ) \to \Omega^1(w^*\xi)$ are
 zero-order differential operators given by
$$
B(Y) =
- \frac{1}{2}  w^*\lambda \left((\CL_{X_\lambda}J)J Y\right)
$$
and
$$
T_{dw}^\pi(Y) = \pi T(Y,dw)
$$
respectively.
\end{thm}
\begin{proof}
Let $Y$ be a vector field over $w$ and $w_s$ be a family of maps $w_s: \Sigma \to Q$ with
$w_0 = w$ and $Y = \frac{d}{ds}\Big|_{s = 0} w^s$, and $a = \frac{d\gamma}{dt}\Big|_{t = 0}$
for a curve $\gamma$ with $\gamma(0) = z$. We decompose
$$
Y = Y^\pi + \lambda(Y)\, X_\lambda
$$
into the sum of $\xi$-component and $X_\lambda$-component.
Now we calculate
\be\label{eq:Dwdpi}
D_w(d^\pi)(Y): = \nabla^\pi_s(\pi dw_s)\Big|_{s=0}
= \pi \nabla_s(\pi dw_s)\Big|_{s=0}
\ee
We will evaluate
\beastar
\nabla^\pi_s(\pi dw_s)& = & \pi \nabla_s (\Pi dw_s) \\
& = & \pi (\nabla_s \Pi)(dw_s) + \pi \nabla_s (dw_s).
\eeastar
To evaluate this, we recall the following basic identity

\begin{lem}[Equations (5.2) \& (5.3) \cite{oh-wang1}] Let $\nabla$ be the
contact triad connection. Then
\be\label{eq:PinabalPi}
\Pi (\nabla \Pi) Y = 0
\ee
for all $Y \in \xi$, and
\be\label{eq:nablaPiX}
(\nabla \Pi) X_\lambda = - \Pi \nabla X_\lambda = - \Pi \left(\frac{1}{2} (\CL_{X_\lambda}J)J\right).
\ee
\end{lem}
Using this lemma, we compute
\bea\label{eq:pinablaPidws}
\pi (\nabla_s \Pi)(dw_s) & = & \pi (\nabla_s \Pi)(d^\pi w_s + w_s^*\lambda\, X_\lambda) \nonumber\\
& = & \pi (\nabla_s \Pi)(w_s^*\lambda\, X_\lambda) = w_s^*\lambda\, \pi (\nabla_s \Pi)(X_\lambda)\nonumber\\
& = & -  w_s^*\lambda\, \pi \left(\frac{1}{2} (\CL_{X_\lambda}J)JY\right).
\eea
Next, the standard computation of $\nabla_s (dw_s)|_{s = 0}$ gives rise to
\bea\label{eq:pinabladws}
\pi \nabla_s (dw_s)|_{s=0}(a) & = & \pi \nabla_s \left(dw_s\left(\frac{d\gamma}{dt}\right)\right)
\Big|_{(s,t) = (0,0)} \nonumber \\
& = & \pi \nabla_s \frac{d}{dt} (w_s\circ \gamma)\Big|_{(s,t) = (0,0)} \nonumber\\
& = & \pi(\nabla_a Y + T(Y,dw(a))\nonumber\\
& = & \pi(\nabla_a Y)  + \pi(T(Y,dw(a)).
\eea
 On the other hand, we compute
\beastar
\pi(\nabla_a Y) & = & \pi (\nabla_a Y^\pi + \nabla_a(\lambda(Y)\, X_\lambda) \\
& = & \nabla_a^\pi Y^\pi  + \lambda(Y) \nabla_a X_\lambda \\
& = & \nabla_a^\pi Y^\pi  + \lambda(Y) \nabla_{d^\pi w(a)} X_\lambda \\
& = & \nabla_a^\pi Y^\pi + \frac{1}{2} \lambda(Y) (\CL_{X_\lambda}J)J d^\pi w(a)
\eeastar
where we used the formula $\nabla X_\lambda = \frac{1}{2} (\CL_{X_\lambda} J)J$ for
the second equality. This proves
$$
\pi(\nabla Y) =  \nabla^\pi Y^\pi + \frac{1}{2} \lambda(Y) (\CL_{X_\lambda}J)J d^\pi w.
$$
Substituting this into \eqref{eq:pinabladws}, we derive
$$
\pi \nabla_s (dw_s)|_{s=0} =
\nabla^\pi Y^\pi + \frac{1}{2} \lambda(Y) (\CL_{X_\lambda}J)J d^\pi w.
$$
Combining this with \eqref{eq:pinablaPidws}, we obtain
$$
\nabla^\pi_s(\pi dw_s)|_{s=0} = \nabla^\pi Y^\pi + T^\pi(Y,dw)+ \frac{1}{2} \lambda(Y) (\CL_{X_\lambda}J)J  \pi d w
-  w^*\lambda \left(\frac{1}{2} (\CL_{X_\lambda}J)J Y\right).
$$
Therefore we have derived
\beastar
D_w(d^\pi)(Y) & = & \nabla^\pi_s(\pi dw_s)|_{s=0} \\
& = & \nabla^\pi Y^\pi + T^\pi(Y,dw) + \frac{1}{2} \lambda(Y) \pi (\CL_{X_\lambda}J)J dw
- \frac{1}{2} w^*\lambda \left((\CL_{X_\lambda}J)J Y\right).
\eeastar
We note that
\beastar
\frac{1}{2} \left(\lambda(Y) (\CL_{X_\lambda}J)J  \pi dw\right)^{(0,1)}
& = & \frac{1}{2} \lambda(Y)\left(\frac{\CL_{X_\lambda} J)J \pi dw + J (\CL_{X_\lambda}J)J \pi dw\circ j}{2}\right)\\
& = & \frac{1}{2} \lambda(Y) \CL_{X_\lambda}J J\left(\frac{\pi dw - J \pi dw\circ j}{2}\right)\\
& = & \frac{1}{2} \lambda(Y) (\CL_{X_\lambda}J)J \del^\pi w
\eeastar
where $\del^\pi w = (\pi dw)^{(1,0)}$.
By taking the $(0,1)$-projection,
we have proved \eqref{eq:Dwdelbarpi}.

Next we compute $D\Upsilon_2(w)$ and prove \eqref{eq:Dwddot}.
We compute $\frac{d}{ds}|_{s= 0} d(w_s^*\lambda \circ j)$
\be\label{eq:nablas}
\frac{d}{ds}\Big|_{s = 0} d(w_s^*\lambda\circ j)
= d\left(\frac{d}{ds}\Big|_{s = 0}w_s^*\lambda \circ j\right).
\ee
By Cartan's formula applied to the \emph{vector field $Y$ over the map $w$},
we obtain
$$
\frac{d}{ds}\Big|_{s = 0} w_s^*\lambda = Y \rfloor d\lambda + d(Y \rfloor \lambda)
$$
where $\rfloor$ is the interior product over the map $w$.
Substituting this into \eqref{eq:nablas}, we derive
\beastar
\frac{d}{ds}\Big|_{s = 0} d(w_s^*\lambda\circ j) & =  &
d(d(\lambda(Y))\circ j) + d((Y \rfloor d\lambda) \circ j) \\
& = & - \Delta (\lambda(Y))\, dA + d((Y \rfloor d\lambda) \circ j).
\eeastar

This proves
\be\label{eq:linearizedUpsilon2}
D\Upsilon_2(w)(Y) = - \Delta (\lambda(Y))\, dA + d((Y \rfloor d\lambda) \circ j)
= - \Delta (\lambda(Y))\, dA + d((Y^\pi \rfloor d\lambda) \circ j)
\ee
which finishes the proof of Theorem \ref{thm:linearization}.
\end{proof}

Now we evaluate the $D\Upsilon_1(w)$
more explicitly. We have
$$
\delbar^{\nabla^\pi}Y = \frac{1}{2} \left(\nabla^\pi Y + J \nabla^\pi_{j(\cdot)} Y \right)
$$
and $B^{(0,1)}(Y)$ becomes
$$
-\frac{1}{4}\left(w^*\lambda\,  \pi ((\CL_{X_\lambda}J)J Y) + w^*\lambda \circ j\, \pi (\CL_{X_\lambda}J)Y \right).
$$

\section{Fredholm theory and index calculations}
\label{sec:fredholm}

We divide our discussion into the closed case and the punctured case.

\subsection{The closed case}

We start with the following classification result. This is stated by Abbas as a part of \cite[Proposition 1.4]{abbas}.
A somewhat different proof is also given in \cite{oh-wang2}. (See Proposition 3.3 \cite{oh-wang2}.)

\begin{prop}\label{prop:classify} Assume $w:\Sigma\to M$ is a smooth contact instanton from a closed Riemann surface.
Then
\begin{enumerate}
\item If $g(\Sigma)=0$, $w$ can only be a constant map;
\item If $g(\Sigma)\geq 1$, $w$ is either a constant or has its locus of its image
is a \emph{closed} Reeb orbit.
\end{enumerate}
In particular, any such instanton is massless and satisfies $[w] = 0$ in $H_2(Q;\Z)$.
\end{prop}

From the expression of the map $\Upsilon = (\Upsilon_1,\Upsilon_2)$, the map defines a bounded linear map
\be\label{eq:dUpsilon}
D\Upsilon(w): \Omega^0_{k,p}(w^*TQ) \to \Omega^{(0,1)}_{k-1,p}(w^*\xi) \oplus \Omega^2_{k-2,p}(\Sigma).
\ee
We choose $k \geq 2, \, p > 2$. Recalling the decomposition
$$
Y = Y^\pi + \lambda(Y)\, X_\lambda,
$$
we have the decomposition
$$
\Omega^0_{k,p}(w^*TQ) \cong \Omega^0_{k,p}(w^*\xi) \oplus \Omega^0_{k,p}(\dot \Sigma,\R)\cdot X_\lambda.
$$
Here we use the splitting
$$
TQ = \span_\R\{X_\lambda\} \oplus \xi
$$
where $\span_\R\{X_\lambda\}: = \CL$ is a trivial line bundle and so
$$
\Gamma(w^*\CL) \cong C^\infty(\Sigma).
$$
By definition as the linearization operator $D\Upsilon_2(w)$ acts trivially for the section
$Y$ tangent to the Reeb direction.

It follows that the map $D\Upsilon(w)$ is a partial differential operator whose
symbol map is given by $\sigma(D\Upsilon) = \sigma(D\Upsilon_1) \oplus \sigma(D\Upsilon_2)$
where
\bea\label{eq:symbol}
\sigma(D\Upsilon_1(w))(\eta) & = & J\Pi^*\eta \nonumber\\
\sigma(D\Upsilon_2(w))(\eta) & = & \langle \lambda,\eta\rangle^2 = (\eta(X_\lambda))^2
\eea
where $\eta$ is a cotangent vector in $T^*Q \setminus \{0\}$ and
has decomposition
$$
\eta = \eta^\pi + \eta(X_\lambda(\pi(\eta))\,\lambda(\pi(\eta)).
$$
Therefore $D\Upsilon(w)$ can be written into the matrix form
\be\label{eq:matrixDUpsilon}
\left(\begin{matrix}\delbar^{\nabla^\pi} + T_{dw}^{\pi,(0,1)} + B^{(0,1)}
 & \frac{1}{2} \lambda(\cdot) (\CL_{X_\lambda}J)J \del^\pi w \\
d\left((\cdot) \rfloor d\lambda) \circ j\right) & -\Delta(\lambda(\cdot)) \,dA
\end{matrix}
\right)
\ee
where
\beastar
\delbar^{\nabla^\pi} + B^{(0,1)}& : & \Omega^0_{k,p}(w^*\xi) \to
\Omega^{(0,1)}_{k-1,p}(w^*\xi)\\
-* \Delta & : & \Omega^0_{k,p}(\Sigma) \to \Omega^2_{k-2,p}(\Sigma)\\
d\left((\cdot) \rfloor d\lambda) \circ j\right) &: & \Omega^0_{k,p}(w^*\xi)
 \to
\Omega^2_{k-1,p}(\Sigma) \hookrightarrow \Omega^2_{k-2,p}(\Sigma).
\eeastar
In particular we note that the restriction $D\Upsilon_1(w)|_{\Omega^0(w^*\xi)}$ has the same
symbol as that of
$$
\delbar^{\nabla^\pi} : \Omega^0(w^*\xi) \to \Omega^{(0,1)}(w^*\xi)
$$
which is the first order elliptic operator of Cauchy-Riemann type, and
$D\Upsilon_2(w)$ has the symbol of the Hodge Laplacian acting on zero forms
$$
*\Delta: \Omega^0(\Sigma) \to \Omega^2(\Sigma).
$$

We now establish Fredholm property and the index formula of the operator $D\Upsilon(w)$
by dividing the study into the closed and the punctured cases.

For the closed case, we derive

\begin{prop}\label{prop:closed-fredholm} Consider the completion of $D\Upsilon(w)$,
which we still denote by $D\Upsilon(w)$, as a bounded linear map
from $\Omega^0_{k,p}(w^*TQ)$ to $\Omega^{(0,1)}(w^*\xi)\oplus \Omega^2(\Sigma)$
for $k \geq 2$ and $p \geq 2$. Then the operator $D\Upsilon(w)$ is homotopic to the operator
\be\label{eq:diagonal}
\left(\begin{matrix}\delbar^{\nabla^\pi} + T_{dw}^{\pi,(0,1)}+ B^{(0,1)} & 0 \\
0 & -\Delta(\lambda(\cdot)) \,dA
\end{matrix}
\right)
\ee
via the homotopy
\be\label{eq:s-homotopy}
s \in [0,1] \mapsto \left(\begin{matrix}\delbar^{\nabla^\pi} + T_{dw}^{\pi,(0,1)} + B^{(0,1)}
& \frac{s}{2} \lambda(\cdot) (\CL_{X_\lambda}J)J (\pi dw)^{(1,0)} \\
s\, d\left((\cdot) \rfloor d\lambda) \circ j\right) & -\Delta(\lambda(\cdot)) \,dA
\end{matrix}
\right) =: L_s
\ee
which is a continuous family of Fredholm operators. And the principal symbol
$$
\sigma(z,\eta): w^*TQ|_z \to w^*\xi|_z \oplus \Lambda^2(T_z\Sigma), \quad 0 \neq \eta \in T^*_z\Sigma
$$
of \eqref{eq:diagonal} is given by the matrix
\beastar
\left(\begin{matrix} \frac{\eta + i\eta \circ j}{2} Id  & 0 \\
0 & |\eta|^2
\end{matrix}\right)
\eeastar
after applying the isomorphism $*: \Omega^2(\Sigma) \to \Omega^0(\Sigma)$ and so is elliptic.
\end{prop}
\begin{proof}
It is enough to establish the inequality
\bea\label{eq:Ls}
\|Y\|_{k,p} & \leq & C (\|\pi_1(L_s(Y))\|_{k-1,p}+\|\pi_1(K_s(Y))\|_{k-1,p})\nonumber\\
&{}& \quad +\|\pi_2(L_s(Y))\|_{k-2,p}+\|\pi_2(K_s(Y))\|_{k-2,p})
\eea
for a family of compact operators $K_s: \Omega^0_{k,p}(w^*TQ)
\to \Omega^{(0,1)}_{k-1,p}(w^*\xi) \oplus \Omega^2_{k-2,p}(\Sigma)$
and a constant $C$ independent of $s \in [0,1]$ for all $Y \in \Omega_{k,p}(w^*TQ)$.

We decompose $Y = Y^\pi + \lambda(Y)\, X_\lambda$.
We have already computed above
\beastar
\pi_1(L_s(Y)) &= &(\delbar^{\nabla^\pi} + T_{dw}^{\pi,(0,1)} + B^{(0,1)})(Y^\pi)
+ \frac{s}{2}\lambda(Y)(\CL_XJ)J (\pi dw)^{(1,0)}\\
\pi_2(L_s(Y)) & = & s\, d(Y \rfloor d\lambda)\circ j - \Delta (\lambda(Y))\,dA.
\eeastar
By the ellipticity of $\delbar^{\nabla^\pi}  + T_{dw}^{\pi,(0,1)} + B^{(0,1)}: \Omega^0(w^*\xi) \to \Omega^{(0,1)}(w^*\xi)$
and of $\Delta: \Omega^0(\Sigma) \to \Omega^0(\Sigma)$, we have
\be\label{eq:|Ypi|}
\|Y^\pi\|_{k,p} \leq C (\|(\delbar^{\nabla^\pi}  + T_{dw}^{\pi,(0,1)}+ B^{(0,1)})(Y^\pi)\|_{k-1,p} + \|Y^\pi\|_{k-1,p})
\ee
and
\be\label{eq:|lambdaY|}
\|\lambda(Y)\|_{k,p} \leq C(\|\Delta(\lambda(Y))\|_{k-2,p} + \|\lambda(Y)\|_{k-2,p}).
\ee
Then we get
\beastar
&{}& \|\lambda(Y)(\CL_XJ)J (\pi dw)^{(1,0)}\|_{k-1,p}\nonumber \\
& \leq & C_k (\|(\CL_XJ)J (\pi dw)^{(1,0)}\|_{k-2,\infty}\|\lambda(Y)\|_{k-1,p} +  \|(\CL_XJ)J (\pi dw)^{(1,0)}\|_{k-1,\infty}\|\lambda(Y)\|_{k-2,p})\\
& \leq & C_k \|(\CL_XJ)J (\pi dw)^{(1,0)}\|_{k-2,\infty}\left(C(\|\Delta(\lambda(Y))\|_{k-2,p}
 + \|\lambda(Y)\|_{k-2,p}\right)
\eeastar
(Here the last line can be improved by $k-3$ for $k \geq 3$ but $k-2$ will be enough
for our purpose which we have to use anyway for $k = 2$),
and
$$
\|d(Y \rfloor d\lambda)\circ j\|_{k-2,p} \leq C_k (\|Y^\pi\|_{k-1,p} \|d\lambda\|_{k-2,\infty} + \|Y^\pi\|_{k-1,p}\|d\lambda\|_{k-1,\infty})
$$
for some constant $C_k$ depending only on $k$ (and $dw$) but independent of $Y$.
Combining all the above, using the bounds for $\|(\CL_XJ)J (\pi dw)^{(1,0)}\|_{k-2,\infty}$ and
$\|d\lambda\|_{k-1,\infty}$ and substituting
$$
(\delbar^{\nabla^\pi} + T_{dw}^{\pi,(0,1)}+ B^{(0,1)})(Y^\pi) = \pi_1(L_s(Y)) -
 \frac{s}{2}\lambda(Y)(\CL_XJ)J (\pi dw)^{(1,0)}
$$
and
$$
- \Delta(\lambda(Y))\, dA = \pi_2(L_s(Y)) - s\, d(Y \rfloor d\lambda)\circ j
$$
into \eqref{eq:|Ypi|} and \eqref{eq:|lambdaY|} and then rearranging terms,
we derive
\be\label{eq:Ykp}
\|Y\|_{k,p} \leq C (\|\pi_1(L_s(Y))\|_{k-1,p}+ \|Y^\pi\|_{k-1,p}+ \|\pi_2(L_s(Y))\|_{k-2,p}
+\|Y\|_{k-2,p})
\ee
for a constant $C$ independent of $s \in [0,1]$ for all $Y \in \Omega_{k,p}(w^*TQ)$. By the
compactness of the Sobolev embedding $W^{l,p}$ into $W^{l-1,p}$ for $l = k, \, k-1$
(on compact $\Sigma$), we have finished the proof of \eqref{eq:Ls} by taking the
operator $K_s = K_{1,s} + K_{2,s}$: Here $K_{1,s}$ is the composition of the bounded map
$$
\Omega^0_{k,p}(w^*TQ) \to \Omega^{(0,1)}_{k,p}(w^*\xi)\oplus
\Omega^2_{k-1,p}(\Sigma)
$$
defined by
$$
Y \mapsto \left(\begin{matrix} \frac{s}{2}\lambda(Y)(\CL_XJ)J (\pi dw)^{(1,0)}\\
s\, d\left((\cdot) \rfloor d\lambda) \circ j\right)
\end{matrix}
\right)
$$
and the inclusion map
$$
\Omega^{(0,1)}_{k,p}(w^*\xi)\oplus
\Omega^2_{k-1,p}(\Sigma) \to \Omega^{(0,1)}_{k-1,p}(w^*\xi)\oplus
\Omega^2_{k-2,p}(\Sigma)
$$
which is compact. In particular, $K_{1,s}$ is a compact operator.

And we define $K_{2,s}$ is just the inclusion map
$$
\Omega^0_{k,p}(w^*TQ) \cong \Omega^0_{k,p}(w^*\xi) \oplus \Omega^0_{k,p}(\Sigma)
\hookrightarrow \Omega^0_{k-1,p}(w^*\xi) \oplus \Omega^0_{k-2,p}(\Sigma)
$$
which is also compact. Obviously
$$
\|Y^\pi\|_{k-1,p} +\|\lambda(Y)\|_{k-2,p} \leq \|\pi_1(K_{2,s}(Y))\|_{k-1,p} + \|\pi_2(K_{2,s}(Y)\|_{k-2,p}.
$$
Therefore combining all the above, we have established \eqref{eq:Ls} which
 finishes the proof.
\end{proof}

From this, we immediately derive the following index formula for $D\Upsilon(w)$ from the homotopy
invariance of the index

\begin{thm}\label{thm:index} Let $\Sigma$ be any closed Riemann surface
of genus $g$, and let $w: \Sigma \to Q$ be a solution to \eqref{eq:contact-instanton}
with finite energy. Then the operator \eqref{eq:dUpsilon} is a Fredholm operator
whose index is given by
\be\label{eq:indexwhen0}
\operatorname{Index} D\Upsilon(w) =  2n(1-g).
\ee
\end{thm}
\begin{proof}

We already know that the operators $\delbar^{\nabla^\pi} +  T_{dw}^{\pi,(0,1)} + B^{(0,1)}$ and
$-\Delta$ are Fredholm. Furthermore we can homotope the operator
\eqref{eq:matrixDUpsilon} to the direct sum operator
$$
(\delbar^{\nabla^\pi}  + T_{dw}^{\pi,(0,1)} + B^{(0,1)} + \frac{1}{2} \lambda(\cdot) (\CL_{X_\lambda}J)J \del^\pi w \oplus (-* \Delta(\lambda(\cdot)))
$$
by considering the continuous deformation of Fredholm operators
$$
s \mapsto \left(\begin{matrix}\delbar^{\nabla^\pi}  + T_{dw}^{\pi,(0,1)} + B^{(0,1)} & \frac{1}{2} \lambda(\cdot) (\CL_{X_\lambda}J)J \del^\pi w \\
s \, d\left((\cdot) \rfloor d\lambda) \circ j\right) & -* \Delta(\lambda(\cdot))
\end{matrix}
\right)
$$
from $s =1$ to $s = 0$.
From this, the Fredholm property immediately follows. Then the index is given by
$$
\operatorname{Index}\delbar^{\nabla^\pi} + \operatorname{Index} (-\Delta) = 2c_1(w^*\xi) + 2n(1-g)
+ 0 = 2c_1(w^*\xi) + 2n(1-g)
$$
in general. But since  $[w] = 0$ in $H_2(Q;\Z)$ by Proposition \ref{prop:classify}, this is reduced to
\eqref{eq:indexwhen0}. This finishes the proof.
\end{proof}

We would like to call attention of readers that the index $\operatorname{Index}\delbar^{\nabla^\pi} = 2n$
when $g = 0$ is $1$ smaller than the dimension of $Q$.

\subsection{The punctured case}

For the punctured case, we need to make some preparation.
For the exposition of this section, we adapt the exposition given by Bourgeois and Mohnke in
\cite{bourg-mohnke} to the current context of contact Cauchy-Riemann maps.
Because the structure of the linearization of \eqref{eq:contact-instanton} is significantly
different, establishing the Fredholm property of the linearization map and its index calculation
is also different. In particular, a priori the ellipticity itself of the linearization map
is not obvious.

From now on in the rest of the paper, we will restrict ourselves to the case of vanishing charge,
i.e., we put the following hypothesis.
\begin{hypo}[Charge vanishing]\label{hypo:exact}
We assume the asymptotic charges of $w$ at all ends vanish, i.e.,
\be\label{eq:asymp-a=0}
-a = \lim_{\tau \to \infty} \int_{\del_\ell \Sigma(\rho)} w(\tau,\cdot)^*\lambda \circ j = 0
\ee
for all $\ell =1, \cdots, k$ where $\rho = e^{-2\pi\tau}$.
\end{hypo}

Let $(\dot \Sigma, j)$ be a punctured Riemann surface and let
$$
p_1, \cdots, p_{s^+}, q_1, \cdots, q_{s^-}
$$
be the positive and negative punctures. Fix an elongation function $\rho: \R \to [0,1]$
so that
\beastar
\rho(\tau) & = & \begin{cases} 1 \quad & \tau \geq 1 \\
0 \quad & \tau \leq 0
\end{cases} \\
0 & \leq & \rho'(\tau) \leq 2.
\eeastar

Let $\gamma^+_i$ for $i =1, \cdots, s^+$ and $\gamma^-_j$ for $j = 1, \cdots, s^-$
be two given collections of Reeb orbits. For each $p_i$ (resp. $q_j$), we associate the isothermal
coordinates $(\tau,t) \in [0,\infty) \times S^1$ (resp. $(\tau,t) \in (-\infty,0] \times S^1$)
on the punctured disc $D_{e^{-2\pi R_0}}(p_i) \setminus \{p_i\}$
(resp. on $D_{e^{-2\pi R_0}}(q_i) \setminus \{q_i\}$) for some sufficiently large $R_0 > 0$.
Then we consider sections of $w^*TQ$ by
\be\label{eq:barY}
\overline Y_i = \rho(\tau - R_0) X_\lambda(\gamma^+_k(t)),\quad
\underline Y_j = \rho(\tau + R_0) X_\lambda(\gamma^+_k(t))
\ee
and denote by $\Gamma_{s^+,s^-} \subset \Gamma(w^*TQ)$ the subspace defined by
$$
\Gamma_{s^+,s^-} = \bigoplus_{i=1}^{s^+} \R\{\overline Y_i\} \oplus \bigoplus_{j=1}^{s^-} \R\{\underline Y_j\}.
$$
Let $k \geq 2$ and $p > 2$. We denote by
$$
\CW^{k,p}_\delta(\dot \Sigma, Q;J;\gamma^+,\gamma^-), \quad k \geq 2
$$
the Banach manifold such that
\be\label{eq:limatinfty}
\lim_{\tau \to \infty}w((\tau,t)_i) = \gamma^+_i(T_i(t+t_i)), \quad
\lim_{\tau \to - \infty}w((\tau,t)_j) = \gamma^-_j(T_j(t-t_j))
\ee
for some $t_i, \, t_j \in S^1$, where
$$
T_i = \int_{S^1} (\gamma^+_i)^*\lambda, \, T_j = \int_{S^1} (\gamma^-_j)^*\lambda.
$$
Here $t_i,\, t_j$ depends on the given analytic coordinate and the parameterization of
the Reeb orbits.

The local model of the tangent space  of $\CW^{k,p}_\delta(\dot \Sigma, Q;J;\gamma^+,\gamma^-)$ at
$w \in C^\infty_\delta(\dot \Sigma,Q) \subset W^{k,p}_\delta(\dot \Sigma, Q)$ is given by
\be\label{eq:tangentspace}
\Gamma_{s^+,s^-} \oplus W^{k,p}_\delta(w^*TQ)
\ee
where $W^{k,p}_\delta(w^*TQ)$ is the Banach space
\beastar
&{} & \{Y = (Y^\pi, \lambda(Y)\, X_\lambda)
\mid e^{\frac{\delta}{p}|\tau|}Y^\pi \in W^{k,p}(\dot\Sigma, w^*\xi), \,
\lambda(Y) \in W^{k,p}(\dot \Sigma, \R)\}\\
& \cong & W^{k,p}(\dot \Sigma, \R) \cdot X_\lambda(w) \oplus W^{k,p}(\dot\Sigma, w^*\xi).
\eeastar
Here we measure the various norms in terms of the triad metric of the triad $(Q,\lambda,J)$.
To describe the choice of $\delta > 0$, we need to recall the covariant linearization of the map $
D\Upsilon_{\lambda, T}: W^{1,2}(z^*\xi) \to L^2(z^*\xi)
$
of the map
$$
\Upsilon_{\lambda,T}: z \mapsto \dot z - T\, X_\lambda(z)
$$
for a given $T$-periodic Reeb orbit $(T,z)$. The operator has the expression
\be\label{eq:DUpsilon}
D\Upsilon_{\lambda, T} = \frac{D^\pi}{dt} - \frac{T}{2}(\CL_{X_\lambda}J) J=: A_{(T,z)}
\ee
where $\frac{D^\pi}{dt}$ is the covariant derivative
with respect to the pull-back connection $z^*\nabla^\pi$ along the Reeb orbit
$z$ and $(\CL_{X_\lambda}J) J$ is (pointwise) symmetric operator with respect to
the triad metric. (See Lemma 3.4 \cite{oh-wang1}.)
We choose $\delta> 0$ so that $0 < \delta/p < 1$ is smaller than the
spectral gap
\be\label{eq:gap}
\text{gap}(\gamma^+,\gamma^-): = \min_{i,j}
\{d_{\text H}(\text{spec}A_{(T_i,z_i)},0),\, d_{\text H}(\text{spec}A_{(T_j,z_j)},0)\}.
\ee

Now for each given $w \in \CW^{k,p}_\delta:= \CW^{k,p}_\delta(\dot \Sigma, Q;J;\gamma^+,\gamma^-)$,
we consider the Banach space
$$
\Omega^{(0,1)}_{k-1,p;\delta}(w^*\xi)
$$
the $W^{k-1,p}_\delta$-completion of $\Omega^{(0,1)}(w^*\xi)$ and form the bundle
$$
\CH^{(0,1)}_{k-1,p;\delta}(\xi) = \bigcup_{w \in \CW^{k,p}_\delta} \Omega^{(0,1)}_{k-1,p;\delta}(w^*\xi)
$$
over $\CW^{k,p}_\delta$. Then we can regard the assignment
$$
\Upsilon_1: w \mapsto \delbar^\pi w
$$
as a smooth section of the bundle $\CH^{(0,1)}_{k-1,p;\delta}(\xi) \to \CW^{k,p}_\delta$. Furthermore
the assignment
$$
\Upsilon_2: w \mapsto d(w^*\lambda \circ j)
$$
defines a smooth section of the trivial bundle
$$
\Omega^2_{k-2,p}(\Sigma) \times \CW^{k,p}_\delta \to \CW^{k,p}_\delta.
$$
We have already computed the linearization of each of these maps in the previous section.

With these preparations, the following is a corollary of exponential estimates established
in Part II \cite{oh-wang2} for the case $Q(p_i) = 0$. We hope that the relevant off-shell analytical framework
for the case $Q(p_i) \neq 0$ can be treated elsewhere.

\begin{prop}[Theorem 1.12 \cite{oh-wang2}]\label{prop:on-containedin-off}
Assume $\lambda$ is nondegenerate and $Q(p_i) = 0$.
Let $w:\dot \Sigma \to Q$ be a contact instanton and let $w^*\lambda = a_1\, d\tau + a_2\, dt$.
Suppose
\bea
\lim_{\tau \to \infty} a_{1,i} = - Q(p_i), &{}& \, \lim_{\tau \to \infty} a_{2,i} = T(p_i)\nonumber\\
\lim_{\tau \to -\infty} a_{1,j} = - Q(q_j), &{}& \, \lim_{\tau \to -\infty} a_{2,j} = T(p_j)
\eea
at each puncture $p_i$ and $q_j$.
Then $w \in \CW^{k,p}_\delta(\dot \Sigma, Q;J;\gamma^+,\gamma^-)$.
\end{prop}

Now we are ready to define the moduli space of contact instantons with prescribed
asymptotic condition as the zero set
\be\label{eq:defn-MM}
\CM(\dot \Sigma, Q;J;\gamma^+,\gamma^-) = \CW^{k,p}_\delta(\dot \Sigma, Q;J;\gamma^+,\gamma^-)
\cap \Upsilon^{-1}(0)
\ee
whose definition does not depend on the choice of $k, \, p$ or $\delta$ as long as $k\geq 2, \, p>2$ and
$\delta > 0$ is sufficiently small. One can also vary $\lambda$ and $J$ and define the universal
moduli space whose detailed discussion is postponed.

In the rest of this section, we establish the Fredholm property of
the linearization map
$$
D\Upsilon_{(\lambda,T)}(w): \Omega^0_{k,p;\delta}(w^*TQ;J;\gamma^+,\gamma^-) \to
\Omega^{(0,1)}_{k-1,p;\delta}(w^*\xi) \oplus \Omega^2_{k-2,p}(\Sigma)
$$
and compute its index. Here we also denote
$$
\Omega^0_{k-2,p;\delta}(w^*TQ;J;\gamma^+,\gamma^-) =
W^{k-2,p}_\delta(w^*TQ;J;\gamma^+,\gamma^-)
$$
for the semantic reason.

For this purpose, we remark that
as long as the set of punctures is non-empty, the symplectic vector bundle
$w^*\xi \to \dot \Sigma$ is trivial. We denote by
$
\Phi: E \to \overline \Sigma \times \R^{2n}
$
and by
$$
\Phi_i^+: = \Phi|_{\del_i^+ \overline \Sigma}, \quad \Phi_j^- = \Phi|_{\del_j^- \overline \Sigma}
$$
its restrictions on the corresponding boundary components of $\del \overline \Sigma$.
Using the cylindrical structure near the punctures,
we can extend the bundle to the bundle $E \to \overline \Sigma$ where $\overline \Sigma$
is the real blow-up of the punctured Riemann surface $\dot \Sigma$.

We then consider the following set
$$
\CS: = \{A: [0,1] \to Sp(2n,\R) \mid 1 \not \in \text{spec}(A(1)), \,
A(0) = id, \, \dot A(0) A(0)^{-1} = \dot A(1) A(1)^{-1} \}
$$
of regular paths in $Sp(2n,\R)$ and denote by $\mu_{CZ}(A)$ the Conley-Zehnder index of
the paths following \cite{robbin-salamon}. Recall that for each closed Reeb orbit $\gamma$ with a fixed
trivialization of $\xi$, the covariant linearization $A_{(T,z)}$ of the Reeb flow along $\gamma$
determines an element $A_\gamma \in \CS$. We denote by $\Psi_i^+$ and $\Psi_j^-$
the corresponding paths induced from the trivializations $\Phi_i^+$ and $\Phi_j^-$ respectively.

We have the decomposition
$$
\Omega^0_{k,p;\delta}(w^*TQ;J;\gamma^+,\gamma^-) =
\Omega^0_{k,p;\delta}(w^*\xi) \oplus \Omega^0_{k,p;\delta}(\Sigma)
$$
and again the operator
$$
D\Upsilon_{(\lambda,T)}(w): \Omega^0_{k,p;\delta}(w^*TQ;J;\gamma^+,\gamma^-) \to
\Omega^{(0,1)}_{k-1,p;\delta}(w^*\xi) \oplus \Omega^2_{k-2,p;\delta}(\Sigma)
$$
can be written into the matrix
\be\label{eq:matrixDUpsilon2}
\left(\begin{matrix}\delbar^{\nabla^\pi} + T^{\pi,(0,1)}_{dw}  + B^{(0,1)}
& \frac{1}{2} \lambda(\cdot) (\CL_{X_\lambda}J)J \del^\pi w \\
d\left((\cdot) \rfloor d\lambda) \circ j\right) & -*\Delta(\lambda(\cdot))
\end{matrix}
\right)
\ee
where
\beastar
\delbar^{\nabla^\pi} + T^{\pi,(0,1)}_{dw}  + B^{(0,1)}& : & \Omega^0_{k,p;\delta}(w^*\xi;J;\gamma^+,\gamma^-) \to
\Omega^{(0,1)}_{k-1,p;\delta}(w^*\xi)\\
-* \Delta & : & \Omega^0_{k,p;\delta}(\Sigma) \to \Omega^2_{k-2,p;\delta}(\Sigma)\\
d\left((\cdot) \rfloor d\lambda) \circ j\right) &: & \Omega^0_{k,p;\delta}(w^*\xi;J;\gamma^+,\gamma^-)
 \to
\Omega^2_{k-1,p;\delta}(\Sigma) \hookrightarrow \Omega^2_{k-2,p;\delta}(\Sigma).
\eeastar

The following proposition can be derived from the arguments used by
Lockhart and McOwen \cite{lockhart-mcowen}. However before applying their
general theory, one needs to pay some preliminary
measure to handle the fact that the order of the operators $D\Upsilon(w)$ are
different depending on the direction of $\xi$ or on that of $X_\lambda$.

\begin{prop}\label{prop:fredholm} Suppose $\delta > 0$ satisfies the inequality
$$
0< \delta < \min\left\{\frac{\text{\rm gap}(\gamma^+,\gamma^-)}{p}, \frac{2\pi}{p}\right\}
$$
where $\text{\rm gap}(\gamma^+,\gamma^-)$ is the spectral gap, given in \eqref{eq:gap},
of the asymptotic operators $A_{(T_j,z_j)}$ or $A_{(T_i,z_i)}$
associated to the corresponding punctures. Then the operator
\eqref{eq:matrixDUpsilon2} is Fredholm.
\end{prop}
\begin{proof} We first note that the operators $\delbar^{\nabla^\pi} + T^{\pi,(0,1)}_{dw}  + B^{(0,1)}$ and
$-\Delta$ are Fredholm: The relevant a priori coercive $W^{k,2}$-estimates for any integer $k \geq 1$
for the derivative $dw$ on the punctured Riemann surface $\dot \Sigma$ with cylindrical metric
near the punctures are established in \cite{oh-wang2} for the operator
$\delbar^{\nabla^\pi}  + T^{\pi,(0,1)}_{dw} + B^{(0,1)}$ and the one for $-\Delta$ is standard.
From this, the standard interpolation inequality establishes the $W^{k,p}$-estimates
for $D\Upsilon(w)$ for all $k \geq 2$ and $p \geq 2$. For readers' convenience, we
provide details in Appendix \ref{append:Fredholm}  which essentially  follow from
\cite{lockhart-mcowen}.

Secondly, it follows that the operator
\eqref{eq:matrixDUpsilon2} can be homotoped to the direct sum operator
$$
(\delbar^{\nabla^\pi} + T^{\pi,(0,1)}_{dw}  + B^{(0,1)}) \oplus (-\Delta)
$$
by considering the continuous deformation of operators
$$
s \mapsto \left(\begin{matrix}\delbar^{\nabla^\pi} + T^{\pi,(0,1)}_{dw}  + B^{(0,1)}
& \frac{s}{2} \lambda(\cdot) (\CL_{X_\lambda}J)J \del^\pi w \\
s \, d\left((\cdot) \rfloor d\lambda) \circ j\right) & -* \Delta(\lambda(\cdot))
\end{matrix}
\right)
$$
from $s =1$ to $s = 0$.
Once these two are established, the proof of the proposition is parallel to that of Proposition
\ref{prop:closed-fredholm}. See Appendix \ref{append:Fredholm}.
\end{proof}

Then by the continuous invariance of the Fredholm index, we obtain
\be\label{eq:indexDXiw}
\operatorname{Index} D\Upsilon_{(\lambda,T)}(w) =
\operatorname{Index} (\delbar^{\nabla^\pi} + T^{\pi,(0,1)}_{dw}  + B^{(0,1)}) + \operatorname{Index}(-\Delta).
\ee
Therefore it remains to compute the latter two indices. For this, we obtain

\begin{thm}\label{thm:indexforDUpsilon} We fix a trivialization
$\Phi: E \to \overline \Sigma$ and denote
by $\Psi_i^+$ (resp. $\Psi_j^-$) the induced symplectic paths associated to the trivializations
$\Phi_i^+$ (resp. $\Phi_j^-$) along the Reeb orbits $\gamma^+_i$ (resp. $\gamma^-_j$) at the punctures
$p_i$ (resp. $q_j$) respectively. Then we have
\bea
\operatorname{Index} (\delbar^{\nabla^\pi} + T^{\pi,(0,1)}_{dw}  + B^{(0,1)}) & = &
n(2-2g-s^+ - s^-) + 2c_1(w^*\xi) + (s^+ + s^-) \nonumber\\
&{}& \quad  + \sum_{i=1}^{s^+} \mu_{CZ}(\Psi^+_i)
- \sum_{j=1}^{s^-} \mu_{CZ} (\Psi^-_j)
\label{eq:Indexdelbarpi}\\
\operatorname{Index} (-\Delta)  & = & \sum_{i=1}^{s^+} m(\gamma^+_i)+ \sum_{j=1}^{s^-} m(\gamma^-_j) -g.
\label{eq:indexDelta}
\eea
In particular,
\bea\label{eq:indexforDUpsilon}
\operatorname{Index} D\Upsilon_{(\lambda,T)}(w) & = & n(2-2g-s^+ - s^-) + 2c_1(w^*\xi)\nonumber\\
&{}& \quad  + \sum_{i=1}^{s^+} \mu_{CZ}(\Psi^+_i)
- \sum_{j=1}^{s^-} \mu_{CZ}(\Psi^-_j)\nonumber \\
&{}& \quad  +
\sum_{i=1}^{s^+} (m(\gamma^+_i)+1) + \sum_{j=1}^{s^-}( m(\gamma^-_j)+1)  - g.
\eea
\end{thm}
\begin{proof} The formula \eqref{eq:Indexdelbarpi} can be immediately derived from
the general formula given in the top of p. 52 of Bourgeois's thesis \cite{bourgeois}:
The summand $(s^+ + s^-)$ comes from the factor $\Gamma_{s^+,s^-}$ in the decomposition
\eqref{eq:tangentspace} which has dimension $s^+ + s^-$.

So it remains to compute the index \eqref{eq:indexDelta}.
We recall that any harmonic function on $\dot \Sigma$ can be written as the
imaginary part of a holomorphic function on $\dot \Sigma$ with the same
orders of zeros and poles respectively. (The converse also holds.) Therefore
to compute the (real) index of $-\Delta$, we consider the Dolbeault complex
$$
0 \to \Omega^0(\Sigma; D) \to \Omega^1(\Sigma;D) \to 0
$$
where $D = D^+ + D^-$ is the divisor associated to the set of punctures
$$
D^+ =  \sum_{i=1}^{s^+}m(\gamma^+_i) p_i, \quad
D^- = \sum_{j=1}^{s^-} m(\gamma^-_j) q_j
$$
where $m(\gamma^+_i)$ (resp. $m(\gamma^-_j)$) is the multiplicity of the Reeb orbit
$\gamma^+_i$ (resp. $\gamma^-_j$). The standard Riemann-Roch formula then gives rise to
the formula for the Euler characteristic
\beastar
\chi(D) & = & \dim_\C H^0(D) - \dim_\C H^1(D) = \operatorname{deg} (D) - g\\
& = &\sum_{i=1}^{s^+} m(\gamma^+_i)+ \sum_{j=1}^{s^-} m(\gamma^-_j) - g.
\eeastar

This finishes the proof.

\end{proof}

\section{Generic transversality under the perturbation of $J$}
\label{sec:generic}

We start with recalling the linearization of the equation $\dot x  = X_\lambda(x)$
along a closed Reeb orbit. Let $z$ be a closed Reeb orbit of period $T > 0$. In other words,
$z: \R \to Q$ is a periodic solution of $\dot z = X_\lambda(z)$ with period $T$, thus satisfying $z(T) = z(0)$.

Denote the Reeb flow $\phi^t= \phi^t_{X_\lambda}$ of the Reeb vector field $X_\lambda$,
we can write $z(t) = \phi^t_{X_\lambda}(z(0))$.
In particular $p:= z(0)$ is a fixed point of the diffeomorphism $\phi^T$.
Further, since $L_{X_\lambda}\lambda = 0$,  the contact diffeomorphism $\phi^T$ induces the isomorphism
$$
\Psi_z : = d\phi^T(p)|_{\xi_p}: \xi_p \to \xi_p
$$
which is the tangent map of the Poincar\'e return map $\phi^T$ restricted to $\xi_p$.

\begin{defn} We say a Reeb orbit with period $T$ is \emph{nondegenerate}
if $\Psi_z:\xi_p \to \xi_p$ with $p = z(0)$ has no eigenvalue 1.
\end{defn}

Denote $\Cont(Q,\xi)$ the set of contact 1 forms with respect to the contact structure $\xi$ and $\CL(Q)=C^\infty(S^1,Q)$
the space of loops $z: S^1 = \R /\Z \to Q$. Let $\CL^{1,2}(Q)$ be the $W^{1,2}$-completion of $\CL(Q)$.
We would like to consider some Banach vector bundle $\CL$ over the Banach manifold
$(0,\infty) \times \CL^{1,2}(Q) \times \Cont(Q,\xi)$ whose fiber at $(T, z, \lambda)$ is given by $L^2(z^*TQ)$.
We consider the assignment
$$
\Upsilon: (T,z,\lambda) \mapsto \dot z - T \,X_\lambda(z)
$$
which is section of $\CL$.

Denote $D$ the covariant derivative. Then we have the following expression of the full linearization.
\begin{lem}\label{lem:full-linearization}
\beastar
d{(T,z, \lambda)}\Upsilon(a,Y, B) = \frac{D Y}{dt} - T DX_\lambda(z)(Y)-aX_\lambda- T \delta_{\lambda}X_\lambda(B),
\eeastar
where $a\in \R$, $Y\in T_z\CL^{1,2}(Q)=W^{1,2}(z^*TQ)$ and $B\in T_\lambda \Cont(Q, \xi)$ and
the last term $\delta_{\lambda}X_\lambda$ is some linear operator.
\end{lem}

By using this full linearization, one can study the generic existence of the contact one-forms which make all
Reeb orbits nondegenerate. We refer to Appendix of \cite{wendl} for its complete proof.
We now assume that $\lambda$ is such a generic contact form.

Now we involve the set $\CJ(Q,\lambda)$ given in \eqref{eq:JJQlambda}.
We study the linearization of the map $\Upsilon^{univ}$ which is the map $\Upsilon$ augmented by
the argument $J \in \CJ(Q,\lambda)$. More precisely, we define
$$
\Upsilon^{univ}(j, w, J) = \left(\delbar_J^\pi w, d(w^*\lambda \circ j)\right)
$$
$\delbar$ at each $(j,w,J) \in \delbar^{-1}(0)$. In the discussion below, we will fix the complex
structure $j$ on $\Sigma$, and so suppress $j$ from the argument of $\Upsilon^{univ}$.

We denote the zero set $(\Upsilon^{univ})^{-1}(0)$ by
$$
\MM(Q,\lambda;\overline \gamma, \underline \gamma) = \left\{ (w,J)
\in \CW^{k,p}_\delta(\dot \Sigma, Q;\overline \gamma, \underline \gamma)) \times \JJ^\ell(Q,\lambda)
\, \Big|\, \Upsilon^{univ}(w, J) = 0 \right\}
$$
which we call the universal moduli space. Denote by
$$
\pi_2: \CW^{k,p}(\dot \Sigma, Q;\overline \gamma, \underline \gamma) \times \JJ^\ell(Q,\lambda) \to
\JJ^\ell(Q,\lambda)
$$
the projection. Then we have
\be\label{eq:MMK}
\MM(J;\overline \gamma, \underline \gamma)= \MM(Q,\lambda,J;\overline \gamma, \underline \gamma)
 = \pi_2^{-1}(J) \cap \MM(Q,\lambda;\overline \gamma, \underline \gamma).
\ee

One essential ingredient for the generic transversality under the perturbation of
$J \in \CJ(Q,\lambda)$ is the usage of the following unique continuation result.
We take a short cut in its proof relating the (local) contact instanton to a (local) pseudoholomorphic
curves in a (local) symplectization exploiting the well-known unique continuation result for
the pseudoholomorphic maps. Here again the closedness
condition $d(w^*\lambda \circ j)$ for the contact instanton map $w$ enters in an
essential way.

\begin{prop}[Unique continuation lemma]\label{prop:unique-conti}
Any non-constant contact Cauchy-Riemann map does not
have an accumulation point in the zero set of $dw$.
\end{prop}
\begin{proof} Suppose to the contrary that there exists a
point $z_0 \in \Sigma$ and a sequence $z  \to z_0$ such that
$dw(z ) = 0$ for all $i$.
Since $w^*\lambda \circ j$ is closed on $\Sigma$, it can be
written as $w^*\lambda \circ j = da$ on a neighborhood
of $z_0$ for some locally defined function $a$.
Then the pair $(w,a)$ defines a pseudo-holomorphic map to
$Q \times \R$. From the equation $w^*\lambda \circ j = da$,
we also have $da(z ) = 0$ too. This implies $z $ are
critical points of the pseudoholomorphic map $(w,a)$ with
$z_0$ as an accumulation point of $z $ which are critical points of
$(w,a)$. Then the unique continuation lemma applied to $(w,a)$
implies $(w,a) \equiv const$ and so $w$ must be constant, a
contradiction to the hypothesis. This finishes the proof.
\end{proof}

The following theorem summarizes the main transversality scheme
needed for the study of the moduli problem of contact instanton map, whose proof
is not very different from that of pseudo-holomorphic curves, once the above
unique continuation result is established, and so omitted.

\begin{thm}\label{thm:trans} Let $0 < \ell < k -\frac{2}{p}$.
Consider the moduli space $\MM(Q,\lambda;\overline \gamma, \underline \gamma)$. Then
\begin{enumerate}
\item $\MM(Q,\lambda;\overline \gamma, \underline \gamma)$ is
an infinite dimensional $C^\ell$ Banach manifold.
\item The projection $\Pi_\alpha =
\pi_2|_{ \MM(Q,\lambda,J;\overline \gamma, \underline \gamma)}:
\MM(Q,\lambda,J;\overline \gamma, \underline \gamma)) \to \JJ^\ell(Q,\lambda)$ is a
Fredholm map and its index is the same as that of $D\Upsilon(w)$
for a (and so any) $w \in  \MM(Q,\lambda,J;\overline \gamma, \underline \gamma)$.
\end{enumerate}
\end{thm}

One should compare this with the corresponding statement for Floer's perturbed Cauchy-Riemann equations
on symplectic manifolds.

\appendix

\section{Proof of energy bound for the case of proper potential}
\label{sec:appendix}

In this appendix, we give the proof of Proposition \ref{prop:proper-energy}.

Since $f$ is assumed to be proper, $f(r) = \pm \infty$
for each puncture $r_\ell$ of $\dot \Sigma$ depending on whether the puncture is positive or
negative.

The proof is entirely similar to the proof of Lemma 5.15 \cite{behwz} verbatim
with replacement of $a$ and the equation $dw^*\lambda \circ j = da$ therein by $f$
and the equation
$$
dw^*\lambda \circ j + \sum_{e \in E(T)} Q(w;e)\, dt_e = df
$$
respectively in our current context. (\emph{We would also like point out that \cite{behwz}
used the letter `$f$' for the map $w$ while our notation $f$ is for the contact instanton potential function
 which corresponds to $a$ in their notation}. This should not confuse the readers, hopefully.)

In a neighborhood $D_\delta(p) \subset \C$ of a given puncture $p$ with analytic coordinate
$z$ centered at $p$ and $C_\delta(p) = \del D_\delta(p)$, with oriented positively
for a positive puncture, and negatively for a
negative puncture. Consider the function
$$
\delta \mapsto \int_{C_\delta(p)} w^*\lambda.
$$
It is increasing and bounded above (resp. decreasing and bounded below), if the puncture is
positive (resp. negative), since $d\lambda \geq 0$ on any contact Cauchy-Riemann map $w$ and
$\int_{D_\delta(p)} dw^*\lambda \leq E^\pi(w) < \infty$. Therefore the integral
$$
\int_{C_\delta(p)} w^*\lambda
$$
has a finite limit as $\delta \to 0$ for all punctures. Now let $\varphi \in \CC$ and let $\varphi_n \in \CC$ such that
$\|\varphi - \varphi_n\|_{C^0} \to 0$ and
$\varphi_n\circ f = 0$ on $D_{\frac{1}{n}}(p)$ for all punctures $p$. Such function exists by
the assumption on properness of potential function $f$. Moreover we can choose $\varphi_n$ so that
$$
\int_{\dot\Sigma}(\varphi_n \circ f)\, df \wedge w^*\lambda = \int_{\dot \Sigma} w^*d(\psi_n w^*\lambda)
- \int_{\dot \Sigma} (\psi_n \circ f) w^*d\lambda,
$$
where $\psi_n(s) = \int_{-\infty}^s \varphi_n(\sigma)\, d\sigma$. Notice that $\psi_n\circ f = 1$ in
$D_{\frac{1}{n}}(p)$ when $p$ is a positive puncture and $\psi_n \circ f = 0$ therein when $p$ is negative.
By Stokes' theorem,
$$
\int_{\dot \Sigma} w^*d(\psi_n \lambda) = \lim_{\delta \to 0} \sum_{\ell^+}
\int_{\del_{\ell^+} \Sigma(\delta)} w^*\lambda
$$
where the sum is taken over all positive punctures $p_{\ell^+}$. Therefore
\beastar
\int_{\dot \Sigma} (\varphi_n \circ f)\, df \wedge w^*\lambda & = & \lim_{\delta \to 0}
\sum_{\ell^+} \int_{\del_{\ell^+} \Sigma(\delta)} w^*\lambda - \int_{\dot \Sigma}(\psi_n \circ f)\, w^*d\lambda \\
& \leq & \lim_{\delta \to 0} \sum_{\ell^+}\int_{\del_{\ell^+}D_\delta(p)} w^*\lambda < C < \infty.
\eeastar
Moreover
$$
\int_{\dot \Sigma} (\varphi_n \circ f)\, df \wedge w^*\lambda \rightarrow \int_{\dot \Sigma}
(\varphi\circ f)\, df \wedge w^*\lambda
$$
as $n \to \infty$, which implies
$$
\int_{\dot \Sigma} (\varphi \circ f)\, df  \leq C,
$$
and so $E(w) \leq E^\pi(w) + C < \infty$. This finishes the proof.

\section{Details of the proof of Corollary \ref{cor:C^1oncylinder}}
\label{append:C^1oncylinder}

Suppose to the contrary that $|dw|_{C^0} = \infty$ and let $z_\alpha$ be a blowing-up
sequence. We denote $R_\alpha = |dw(z_\alpha)| \to \infty$. Then by applying Lemma \ref{lem:Hoferslemma},
we can choose another such sequence $z_\alpha'$ and $\epsilon_\alpha \to 0$ such that
$$
|dw(z_\alpha')| \to \infty, \quad \max_{z \in D_{\epsilon_\alpha(z_\alpha')}}|dw(z)|
\leq 2 R_\alpha,
\quad \epsilon_\alpha R_\alpha \to 0.
$$
We consider the re-scaling maps $\widetilde w_\alpha: D^2_{\epsilon_\alpha R_\alpha}(0)
\to Q$ defined by
$$
w_\alpha(z) = w \left(z_\alpha' + \frac{z}{R_\alpha}\right)
$$
where we may identify $D_{\epsilon_\alpha(z_\alpha')}$ as a subset of $\R \times S^1$
for all sufficiently large $\alpha$ since $\epsilon_\alpha \to 0$ as $\alpha \to \infty$.
By the exactly same argument as that of the proof of Theorem \ref{thm:C1bound},
we obtain a contact instanton $w_\infty : \C \to Q$ satisfying
$$
\begin{cases}
E^\pi(w_\infty) = 0, \, d(w_\infty^*\lambda \circ j) = 0, \quad E(w_\infty) < \infty\\
 |dw|_{C^0} \leq 2 < \infty, \, |dw_\infty(0)| = 1.
 \end{cases}
$$
Then the first line of this equation implies that $w_\infty$ is a constant map
by Proposition \ref{prop:C^1}, which obviously contradicts to the equation
$|dw_\infty(0)| = 1$ in the second line.
This finishes the proof of $|dw|_{C^0} < \infty$ and hence the proof.

\section{Details of the Fredholmness proof in Proposition \ref{prop:fredholm}}
\label{append:Fredholm}

We will prove the uniform Fredholm property of the one-parameter family of
operators $L_s$ for  $s \in [0,1]$ as in the proof of Proposition \ref{prop:closed-fredholm}.

Again it is enough to establish that there exists a family of compact operators
$$
K_s: \Omega^0_{k,p;\delta}(w^*TQ)
\to \Omega^{(0,1)}_{k-1,p;\delta}(w^*\xi) \oplus \Omega^2_{k-2,p;\delta}(\Sigma)
$$
such that the inequality
\bea\label{eq:Ls-delta}
\|Y\|_{k,p;\delta} & \leq & C (\|\pi_1(L_s(Y))\|_{k-1,p;\delta}+\|\pi_1(K_s(Y))\|_{k-1,p;\delta})\nonumber\\
&{}& \quad +\|\pi_2(L_s(Y))\|_{k-2,p;\delta}+\|\pi_2(K_s(Y))\|_{k-2,p;\delta})
\eea
holds for all $Y \in \Omega_{k,p;\delta}(w^*TQ)$ for a constant $C> 0$
independent of $s \in [0,1]$ and $w$.

In the discussion henceforth, the constant $C> 0$ may vary but can be always chosen
uniformly which is independent of $s$ and $w$.

We decompose $Y = Y^\pi + \lambda(Y)\, X_\lambda$ as before and recall
the formulae
\beastar
\pi_1(L_s(Y)) &= &(\delbar^{\nabla^\pi} + T_{dw}^{\pi,(0,1)} + B^{(0,1)})(Y^\pi)
+ \frac{s}{2}\lambda(Y)(\CL_XJ)J (\pi dw)^{(1,0)}\\
\pi_2(L_s(Y)) & = & s\, d(Y \rfloor d\lambda)\circ j - \Delta (\lambda(Y))\,dA
\eeastar
from the proof of Proposition \ref{prop:closed-fredholm}.

At this point, we briefly recall the general  a priori estimates for the elliptic operators
in the setting of manifolds with cylindrical ends laid out in
\cite[Section I-1 \& II-8]{lockhart-mcowen}. Let $X$ be a noncompact manifold
possibly with multiple ends. Let
$$
E = \oplus_{j=1}^J E_j, \quad F = \oplus_{i=1}^I F_i.
$$
Let $t = (t_1,\ldots, t_J)$ and $s = (s_1, \ldots, s_I)$ be sets of nonnegative integers and
define
$$
W_{p,t;\delta}(E) = \oplus_{j=1}^J W_{p,t_j;\delta}(E_j), \quad
W_{p,s;\delta}(E) = \oplus_{i=1}^I W_{p,s_ji\delta}(F_i).
$$
A differential operator $A: C_0^\infty(E) \to C_0^\infty(F)$ decomposes into
$$
A_{ij}: C_0^\infty(E_j) \to C_0^\infty(F_i).
$$
If each $A_{ij}$ is of order $t_j = s_i$ (where $t_j - s_i < 0$ implies $A_{ij} = 0$),
then $(t,s)$ is called a \emph{system of orders for $A$}. WLOG, we may assume
that each $t_j > 0$. Assuming that $A$ is translation invariant in the cylindrical ends, we
find that
\be\label{eq:A}
A: W_{p,t;\delta}(E) \to W_{p,s;\delta}(F)
\ee
is a bounded operator.  Then we have

\begin{thm}[Theorem 1.1 \cite{lockhart-mcowen}]\label{thm:lockhart-mcowen}
If $A$ is elliptic with respect to
$(t,s)$ and it is translation invariant on the cylindrical ends, then there is a discrete
set $\CD_A \subset \R$ such that the operator \eqref{eq:A} is Fredholm if and only if
$\delta \in \R \subset \CD_A$.
\end{thm}

\begin{rem}
In the case of current interest, we take $X = \dot \Sigma$ equipped with
the K\"ahler metric $h$ of the Riemann surface $(\dot \Sigma, j)$ that is
cylindrical near punctures, and
$$
E = w^*TQ, \quad F = \Lambda^{(0,1)}_J(w^*\xi) \oplus \Lambda^2(\dot \Sigma).
$$
\end{rem}
Therefore we can apply this theorem directly to $D\Upsilon_{(\lambda, T)}(w)$
to get the relevant Fredholm property for our problem. Alternatively we may more intuitively
apply the theorem only for the cases $J = 1 = I$ separately to each of
the two diagonal components of $D\Upsilon_{(\lambda, T)}(w)$ which is
$$
\left(\begin{matrix}\delbar^{\nabla^\pi} + T^{\pi,(0,1)}_{dw}  + B^{(0,1)}
& 0 \\
0 & -* \Delta(\lambda(\cdot))
\end{matrix}
\right).
$$
 Then by the ellipticities of
$$
\delbar^{\nabla^\pi}  + T_{dw}^{\pi,(0,1)} + B^{(0,1)}:
\Omega^0(w^*\xi) \to \Omega^{(0,1)}(w^*\xi)
$$
and of
$$
\Delta: \Omega^0(\Sigma) \to \Omega^0(\Sigma).
$$
Theorem \ref{thm:lockhart-mcowen} applied to these two cases  separately imply
\be\label{eq:|Ypi|-delta}
\|Y^\pi\|_{k,p;\delta} \leq
C (\|(\delbar^{\nabla^\pi}  + T_{dw}^{\pi,(0,1)}+ B^{(0,1)})(Y^\pi)\|_{k-1,p;\delta}
+ \|Y^\pi\|_{k-1,p;\delta})
\ee
and
\be\label{eq:|lambdaY|-delta}
\|\lambda(Y)\|_{k,p;\delta} \leq C(\|\Delta(\lambda(Y))\|_{k-2,p;\delta} + \|\lambda(Y)\|_{k-2,p;\delta}).
\ee
Now the rest of the argument is exactly the same as that of Proposition \ref{prop:closed-fredholm}
with the Sobolev space $W_{k,p}$ and etc replaced by the weighted ones $W_{k,p;\delta}$.
This finishes the proof of uniform Fredholmless of the family $L_s$ with $s \in [0,1]$ of
operators. We note that the case at $s = 1$ is nothing but the linearized operator at $w$.

\end{document}